\documentclass[11pt,a4paper]{amsart}
\usepackage{amsmath,amssymb}
\usepackage{mathrsfs}
\usepackage{amsrefs}
\usepackage{graphicx,color}
\usepackage{amsthm}
\usepackage{latexsym}
\usepackage[all]{xy}
\usepackage{hyperref}
\usepackage[mathscr]{eucal}
\usepackage{comment}

\usepackage{tikz-cd}
\usepackage{fancyhdr}

\makeatletter
\@addtoreset{equation}{section}
\makeatother

\BibSpec{book}{%
    +{}  {\PrintPrimary}                {transition}
    +{.} { \textit}                     {title}
    +{.} { }                            {part}
    +{:} { \textit}                     {subtitle}
    +{,} { \PrintEdition}               {edition}
    +{}  { \PrintEditorsB}              {editor}
    +{,} { \PrintTranslatorsC}          {translator}
    +{,} { \PrintContributions}         {contribution}
    +{,} { }                            {series}
    +{,} { \voltext}                    {volume}
    +{,} { }                            {publisher}
    +{,} { }                            {organization}
    +{,} { }                            {address}
    +{,} { }                            {status}
    +{,} { \PrintDOI}                   {doi}
    +{,} { \PrintISBNs}                 {isbn}
    +{}  { \parenthesize}               {language}
    +{}  { \PrintTranslation}           {translation}
    +{;} { \PrintReprint}               {reprint}
    +{,} { \PrintDate}                  {date}
    +{.} { }                            {note}
    +{.} {}                             {transition}
}
\BibSpec{article}{%
    +{}  {\PrintAuthors}                {author}
    +{,} { \textit}                     {title}
    +{.} { }                            {part}
    +{:} { \textit}                     {subtitle}
    +{,} { \PrintContributions}         {contribution}
    +{.} { \PrintPartials}              {partial}
    +{,} { }                            {journal}
    +{}  { \textbf}                     {volume}
    +{}  { \PrintDatePV}                {date}
    +{,} { \issuetext}                  {number}
    +{,} { \eprintpages}                {pages}
    +{,} { }                            {status}
    +{,} { \PrintDOI}                   {doi}
    +{}  { \parenthesize}               {language}
    +{}  { \PrintTranslation}           {translation}
    +{;} { \PrintReprint}               {reprint}
    +{.} { }                            {note}
    +{,} { \eprint}                     {eprint}
    +{.} {}                             {transition}
}
\BibSpec{collection.article}{%
    +{}  {\PrintAuthors}                {author}
    +{,} { \textit}                     {title}
    +{.} { }                            {part}
    +{:} { \textit}                     {subtitle}
    +{,} { \PrintContributions}         {contribution}
    +{,} { \PrintConference}            {conference}
    +{}  {\PrintBook}                   {book}
    +{,} { }                            {booktitle}
    +{,} { \PrintDateB}                 {date}
    +{,} { pp.~}                        {pages}
    +{,} { }                            {publisher}
    +{,} { }                            {organization}
    +{,} { }                            {address}
    +{,} { }                            {status}
    +{,} { \PrintDOI}                   {doi}
    +{,} { \eprint}        {eprint}
    +{}  { \parenthesize}               {language}
    +{}  { \PrintTranslation}           {translation}
    +{;} { \PrintReprint}               {reprint}
    +{.} { }                            {note}
    +{.} {}                             {transition}
}

\BibSpec{misc}{
  +{}{\PrintAuthors}  {author}
  +{,}{ \textit}      {title}
  +{,}{ }             {note}
  +{,}{ }             {date}
  +{.}{}              {transition}
}

\theoremstyle{definition}
\newtheorem{dfn}[equation]{Definition}
\theoremstyle{plain}
\newtheorem{thm}[equation]{Theorem}
\newtheorem{prop}[equation]{Proposition}
\newtheorem{lem}[equation]{Lemma}
\newtheorem{cor}[equation]{Corollary}

\theoremstyle{remark}
\newtheorem{rmk}[equation]{Remark}

\DeclareMathOperator{\id}{id}

\newcommand{\fa}{\mathfrak{a}}
\newcommand{\fb}{\mathfrak{b}}
\newcommand{\pt}{\mathrm{pt}}
\newcommand{\cP}{\mathcal{P}}
\newcommand{\cT}{\mathcal{T}}
\newcommand{\cG}{\mathcal{G}}
\newcommand{\cK}{\mathcal{K}}
\newcommand{\bA}{\mathbb{A}}

\newcommand{\bC}{\mathbb{C}}
\newcommand{\bR}{\mathbb{R}}

\newcommand{\bK}{\mathbb{K}}
\newcommand{\bM}{\mathbb{M}}
\newcommand{\bN}{\mathbb{N}}
\newcommand{\bQ}{\mathbb{Q}}
\newcommand{\bZ}{\mathbb{Z}}

\newcommand{\cN}{\mathcal{N}}
\newcommand{\cO}{\mathcal{O}}

\newcommand{\fp}{\mathfrak{p}}
\newcommand{\fq}{\mathfrak{q}} 
\newcommand{\fX}{\mathfrak{X}}
\newcommand{\sB}{\mathscr{B}}
\newcommand{\K}{\mathrm{K}}%K-group
\newcommand{\KK}{\mathrm{KK}}%KK-group
\newcommand{\val}{\mathrm{val}}
\newcommand{\Cl}{\mathit{Cl}}
\newcommand{\lwedge}{{\textstyle\bigwedge}}
\DeclareMathOperator{\Img}{\mathrm{Im}}

\DeclareMathOperator{\Hom}{Hom}

\DeclareMathOperator{\Prim}{Prim}
\DeclareMathOperator{\hotimes}{\hat{\otimes}}
\newcommand{\Val}{\mathfrak{val}}
\title[Reconstructing the Bost--Connes semigroup actions]{Reconstructing the Bost--Connes semigroup actions from $\K$-theory}
\date{\today}

\author[Y. Kubota]{Yosuke Kubota}
\address{iTHEMS Research Group, RIKEN, 2-1 Hirosawa, Wako, Saitama 351-0198, Japan}
\email{yosuke.kubota@riken.jp}

\author[T. Takeishi]{Takuya Takeishi}
\address{Kyoto Institute of Technology, 606-8585, Japan}
\email{takeishi@kit.ac.jp}

\begin{document}
\maketitle
\begin{abstract}
We complete the classification of Bost--Connes systems. 
We show that two Bost--Connes C*-algebras for number fields are isomorphic if and only if the original semigroups actions are conjugate.
Together with recent reconstruction results in number theory by Cornelissen--de Smit--Li--Marcolli--Smit, we conclude that two Bost--Connes C*-algebras are isomorphic if and only if the original number fields are isomorphic. 
\end{abstract}

\section{Introduction}
The Bost--Connes system $(A_K,\sigma_{K,t})$ is a C*-dynamical system attached to a number field $K$.
A specific feature of this system is that its dynamics, in particular the behavior of its KMS-states, reflects the arithmetics of the original number field. For example, the set of extremal KMS-states at low temperature equips a free transitive action of the Galois group $G_K^{\mathrm{ab}}$. At high temperature there is a unique KMS-state. The critical temperature $\beta =1$ is nothing but the critical point of the Dedekind zeta function, which is the partition function of the Bost--Connes system.

After the pioneering work by Bost--Connes \cite{MR1366621} in the case of $\bQ$, the generalization of this dynamical system to an arbitrary number field has been a leading problem in the study of Bost--Connes systems. 
It was completed after a 15-year effort by many mathematicians such as Ha--Paugam \cite{MR2199962} (the definition of the Bost--Connes system), Laca--Larsen--Neshveyev \cite{MR2473881} (the KMS-classification) and Yalkinoglu \cite{MR3010380} (construction of the arithmetic subalgebra). 

For the definition of the Bost--Connes system, one starts with an action of the semigroup $I_K$ of integral ideals of $K$ on a compact space $Y_K$, which is defined by using the Artin reciprocity map in class field theory. It associates a groupoid $\cG_K$ and the Bost--Connes C*-algebra $A_K$ is the groupoid C*-algebra $C^*_r(\cG_K)$. The $\bR$-action $\sigma _K$ is induced from the absolute norm function $N \colon I_K \to \bR_{\geq 0}$. See Section \ref{section:pre} for more details.

According to the philosophy of anabelian geometry, it is natural to expect that the Bost--Connes system remembers the original number field. More precisely, two Bost--Connes systems $A_K$ and $A_L$ have been conjectured to be isomorphic if and only if the number fields $K$ and $L$ are isomorphic. 
Indeed, as is mentioned above, the dynamics of the Bost--Connes system such as the KMS states recovers many of the arithmetics of the number field. 
For example, an isomorphism of Bost--Connes systems immediately induces a bijection between abelianized Galois groups. 
This problem was first considered in \cite{mathNT10090736} and recently a remarkable partial solution is given by Cornelissen, de Smit, Li, Marcolli and Smit in \cite{mathNT170604515,mathNT170604517}, where it is proved that the Bost--Connes semigroup actions $Y_K \curvearrowleft I_K$ and $Y_L \curvearrowleft I_L$ are conjugate if and only if $K$ is isomorphic to $L$. 

Alongside these results, recent works by the second author \cite{MR3554839,MR3545946} provide a new perspective on this problem: Even if we forget the $\bR$-action $\sigma_{K,t}$, the underlying Bost--Connes C*-algebra $A_K$ has rich information. 
A key observation is that $A_K$ has the canonical structure of a C*-algebra over $2^{\cP_K}$ (see Subsection \ref{section:2.3} for more details). 
In particular, the main theorem of \cite{MR3545946} clarifies that we can reconstruct the Dedekind zeta function of $K$ from the ordered $\K_0$-group of irreducible sub-quotients without using the $\bR$-action and KMS-states. 

In this paper, we establish this idea in a complete way. Our main theorem is the following: 

\begin{thm}\label{thm:main}
Let $K$ and $L$ be number fields. The following are equivalent:
\begin{enumerate}
\item The semigroup actions $Y_K \curvearrowleft I_K$ and $Y_L \curvearrowleft I_L$ are conjugate.
\item The Bost--Connes groupoids $\cG_K$ and $\cG_L$ are isomorphic.
\item The Bost--Connes systems $(A_K,\sigma_{K,t})$ and $(A_L,\sigma_{L,t})$ are $\bR$-equivariantly isomorphic.
\item The Bost--Connes C*-algebras $A_K$ and $A_L$ are isomorphic. 
\item There is a bijection $\cP_K \cong \cP_L$ (which identifies $2^{\cP_K} \cong 2^{\cP_L} $) and an ordered $\KK(2^{\cP_K})$-equivalence between $A_K$ and $A_L$.
 \item There is a bijection $\cP_K \cong \cP_L$ and a family of ordered isomorphisms 
\[ \varphi ^F \colon \K_*(B_K^F) \to \K_*(B_L^F)\]
such that  $\varphi ^{F \cup \{ \fp \}} \circ \partial_K^{F,\fp} = \partial_L^{F,\fp} \circ \varphi ^F$ for any finite subset $F \subset \cP_K \cong \cP_L$.
\end{enumerate}
\end{thm}
Here we say that an isomorphism of $\K_0$-groups is ordered if it gives a bijection of the positive cones (see for example \cite[Section 6]{MR1656031}). The precise meaning of (5) is that there is an invertible element in $\KK (2^{\cP_K};A_K,A_L)$ (\cite[Definition 4.1]{MR1796912}, see also \cite[Definition 3.1]{MR2545613}) which induces a family of ordered isomorphisms between filtered $\K_0$-groups \cite[Definition 2.4]{MR2953205}.  
The C*-algebra $B_K^F$ and the homomorphism $\partial _K^{F,\fp}$ are defined in Definition \ref{dfn:B} and Definition \ref{dfn:D} respectively.  

The essential step is (6)$\Rightarrow$(1). For the proof, we essentially give a reconstruction procedure of the topological space $Y_K$ and the action of $I_K$ from the given $\K$-theoretic data. 

Combining Theorem \ref{thm:main} with recent results \cite[Theorem 3.1]{mathNT170604515} and \cite[Theorem 3.1]{mathNT170604517} mentioned above, we complete the classification of Bost--Connes systems and the underlying C*-algebras.
\begin{cor}
Let two number fields $K$ and $L$ satisfy one of the equivalent conditions (1) - (6) in Theorem \ref{thm:main}. Then, $K$ is isomorphic to $L$.
\end{cor}

Theorem \ref{thm:main} can be viewed as not only a classification result of C*-algebras, but also a construction of an invariant of number fields. Actually, the implication (6)$\Rightarrow$(1) means that the family of ordered $\K_*$-groups with boundary homomorphisms provides a complete invariant of number fields. 
It would be interesting to relate this invariant with other known invariants in number theory.

This paper is organized as follows. In Section \ref{section:2}, we revisit the global structure of the Bost--Connes semigroup action and Bost--Connes C*-algebra from the viewpoint of the valuation map. In Section \ref{section:3}, we give a reconstruction procedure of profinite completions of free abelian groups from the $\K_*$-group of the crossed product. Finally, the proof of Theorem \ref{thm:main} is given in Section \ref{section:4}. In Appendix \ref{section:appendix}, we introduce a more direct reconstruction of the profinite completion provided for the authors by Xin Li.

\subsection*{Acknowledgment}
The authors are gratefully indebted to Xin Li for his interest and allowing them to expose his elegant alternative proof in this paper. 
They also would like to thank Yuki Arano, Kazuki Tokimoto and Makoto Yamashita for for helpful discussions. In addition, the second author would like to thank Koji Fujiwara for the support when the second author belonged to the Department of Mathematics of Kyoto University and Research Institute for Mathematical Sciences, Kyoto University. 
The first author is partially supported by MEXT's Program for Leading Graduate Schools, the Research Fellowship of the JSPS (No.\ 26-7081) and JSPS KAKENHI Grant Number JP17H06461.
The second author is supported by JSPS KAKENHI Grant Number JP17H06785.

\section{Structure of the Bost--Connes semigroup actions}\label{section:2}

\subsection{Preliminaries}\label{section:pre}
First of all, we give a quick review of the definition of the Bost--Connes system and related objects. Throughout this paper, we write $\bN$ for the additive semigroup $\{ n \in \bZ \mid n \geq 0\}$.

We start with some notational conventions in algebraic number theory. We basically follow the notations of \cite{MR1697859}. For a commutative ring $R$, we use the symbol $R^*$ for the set of unit elements. For a family of locally compact spaces $\{ X_i \}_{i \in I}$ and compact open subspaces $Y_i  \subset X_i$, the restricted direct product is defined to be
\[ \prod \nolimits _{i \in I}' (X_i,Y_i):= \{ (x_i) \in \prod _{i \in I} X_i \mid \text{ $x_i \not \in Y_i$ for only finitely many $i$'s}  \}.\]

Let $K$ be a number field with the integer ring $\cO_K$. We write $\cP_K$ for the set of prime ideals of $\cO_K$. Let $I_K$ denote the set of nonzero integral ideals of $K$, which forms a semigroup by the multiplication. By the unique prime factorization, it is isomorphic to the free abelian group with the basis $\cP_K$, i.e., $I_K \cong \bigoplus_{\fp \in \cP_K} \fp ^{\bN}$. Similarly, the group $J_K$ of fractional ideals of $K$ is isomorphic to the direct sum $\bigoplus_{\fp \in \cP_K} \fp^{\bZ}$.  

For each $\fp \in \cP_K$, the corresponding local field $K_\fp$ has the integer ring $\cO_\fp$ and the quotient $K_\fp^*/\cO_\fp^*$ is isomorphic to $\bZ$ through the valuation $v_\fp$. Let $\hat{\cO}_K$ denote the ring $\prod _{\fp \in \cP_K}\cO_\fp$,  let $\bA_{K,f}$ denote the ring of finite adeles $\prod '_{\fp \in \cP_K} (K_\fp,\cO_\fp)$ and set $\hat{\cO}_K^\natural := \bA_{K,f}^* \cap \hat{\cO}_K$. A finite idele $a \in \bA_{K,f}^*$ generates a fractional ideal $(a) \in J_K$ and this correspondence induces an isomorphism $\bA_{K,f}^*/\hat{\cO}_K^* \cong J_K$. It restricts to an isomorphism $\hat{\cO}_K^\natural/\hat{\cO}_K^* \cong I_K$.

The group $\bA_{K,f}^*$ of finite ideles acts on $G_K^{\mathrm{ab}}$ through the Artin reciprocity map
\[ \phi _K \colon \bA_{K,f}^* \to G_K^{\mathrm{ab}}. \]
Let $\bar{\phi}_K(g):=\phi_K(g)^{-1}$. The group $\bA_{K,f}^*$ also acts on $\bA_{K,f}$ by the multiplication, which is denoted by $\alpha$. Now, the product action $\alpha \times \bar{\phi}_K$ of $\bA_{K,f}^*$ onto $\bA_{K,f} \times G_K^{\mathrm{ab}}$ induces a $J_K$-action on the quotient space
\[ X_K:= \bA_{K,f} \times _{\hat{\cO}_K^*} G^{\mathrm{ab}}_K=(\bA_{K,f} \times G_K^{\mathrm{ab}})/(\alpha \times \bar{\phi}_K)(\hat{\cO}_K^*). \]
This action is explicitly written as $\fa \cdot [b, \gamma ]:=[ba , \phi_K(a)^{-1}\gamma  ]$ for $b \in \bA_{K,f}$ and $\gamma \in G_K^{\mathrm{ab}}$, where $a$ is a finite idele with $\fa=(a)$.
It restricts to an $I_K$-action on 
\[Y_K:= \hat{\cO}_K \times _{\hat{\cO}_K^*} G^{\mathrm{ab}}_K \subset X_K. \]

The \emph{Bost--Connes groupoid} attached to $K$ is the semigroup transformation groupoid
\[ \cG_K := Y_K \rtimes I_K. \]
More precisely, 
$\cG_K$ is the subgroupoid of transformation groupoid $X_K \rtimes J_K$ associated to the group action $X_K \curvearrowleft J_K$ (in the sense of \cite[Example 5.6.3]{MR2391387}) defined by 
\[ \cG_K=\{ (x , \fa ) \in X_K \rtimes J_K \mid \text{$x$ and $\fa \cdot x$ are in $Y_K$} \}. \]
Note that the target space $\cG_K^0=Y_K$ is compact.

The \emph{Bost--Connes C*-algebra} is the associated groupoid C*-algebra $A_K:=C^*_r(\cG_K) $ (see for example \cite[Section 5.6]{MR2391387}). In other words, it is the corner subalgebra
\[ A_K = 1_{Y_K}(C_0(X_K) \rtimes J_K) 1_{Y_K} \]
of the crossed product $C_0(X) \rtimes J_K$, where $1_{Y_K} \in C_0(X_K)$ is the characteristic function on $Y_K$. 

The dual action of the Pontrjagin dual group $\hat{J}_K \cong (\bR/\bZ)^\infty$ on $C_0(X) \rtimes J_K$ has the property that its restriction to $C_0(X)$ is trivial. Hence it restricts to an action on $A_K$.
The $\bR$-action $\sigma _{K,t}$ on $A_K$ is defined as the composition of this $\hat{J}_K$-action with the dual homomorphism $\bR \cong \hat{\bR} \to \hat{J}_K$ of 
\[ -\log N_K({}\cdot {}) \colon J_K \to \bR , \]
where $N_K$ is the absolute norm function. That is, $\sigma_{K,t}$ is determined by
\[ \sigma_{K,t} ( f u_{\mathfrak{a}})=  N_K(\mathfrak{a})^{it}f u _\fa \]
for $t \in \bR$, $f  \in C(Y_K)$ and $\fa \in I_K$.

\subsection{A decomposition of $Y_K$}\label{section:2.2}
In previous works \cite{MR3554839,MR3545946} of the second author, the structure of the closures of orbits of $\cG _K$ were studied in order to determine the primitive ideal space $\Prim (A_K)$. 
Here we revisit them from the viewpoint of the valuation maps.

Let $\tilde{\bN}$ and $\tilde{\bZ}$ denote $\bN \cup \{ +\infty\}$ and $\bZ \cup \{+\infty \}$ with the order topology respectively. Let $Y_\val:=\tilde{\bN}^{\cP_K}$, on which the semigroup $I_K \cong \bigoplus_{\fp \in \cP_K} \fp^\bN$ acts by the product action of  $\tilde{\bN} \curvearrowleft \bN$. Similarly we define the $J_K$-space $X_\val:=\prod '_{\fp \in \cP_K} (\tilde{\bZ},\tilde{\bN}) $.
Since the valuation $v_\fp \colon K_\fp \to \tilde{\bZ}$ is invariant under the action by $\cO_\fp^*$ on the domain, the composition 
\[ \big( \prod \nolimits_{\fp \in \cP_K} v_\fp \big) \circ \mathrm{pr}_1 \colon \bA_{K,f} \times G_K^{\mathrm{ab}} \to X_\val,  \]
where $\mathrm{pr}_1$ denotes the projection onto the first factor, induces a $J_K$-equivariant proper continuous map $\val _K \colon X_K \to X_\val$. It restricts to an $I_K$-equivariant map $Y_K \to Y_\val$.

The orbit space $X_\val/J_K$ with the quotient topology is isomorphic to the power set $2^{\cP_K}$ with the power-cofinite topology (that is, the family of subsets of the form
\[  U_F:=\{ C \in 2^{ \cP_K} \mid C \cap F=\emptyset \}, \] 
where $F$ runs over all finite subsets of $\cP_K$, forms an open basis of $2^{\cP_K}$) by the map
\[ J_K \cdot (n_\fp) \mapsto \{ \fp \in \cP_K \mid n_\fp =\infty \}. \]
This is because the power-cofinite topology is nothing but the product topology of $2=\{ 0,1\}$ with the topology $\{ \emptyset , \{0\}, 2\}$. 

For a subset $S \subset \cP_K$, we write $X_\val^S$ for the orbit corresponding to the complement $S ^c$. That is, 
\[ X_\val ^S:= J_K \cdot x_0^S \subset X_\val, \]
where 
\[(x_0^S)_\fp = \left\{ \begin{array}{ll} 0 & \text{if $\fp \in S$,}\\
+\infty & \text{if $\fp \not \in S$.}
\end{array} \right. \]
We remark that the closure $\overline{X_\val^S}$ is equal to the union $\bigsqcup_{S' \subset S} X_\val ^{S'}$, which corresponds to the fact that
\[ \overline{\{ S^c \}}= \{ C \in 2^{\cP_K} \mid S^c \subset C \} \subset 2^{\cP_K}. \]
The stabilizer subgroup of $x_0^S$ is
\[ J_K^S := \bigoplus \nolimits_{\fp \not \in S} \fp ^\bZ \subset J_K \]
and the fiber $\val^{-1}_K(x_0^S)$ is canonically isomorphic to
\[ G_K^{S}:=G_K^{\mathrm{ab}}/\big( \prod\nolimits_{\fp \not \in S} \cO_\fp^*\big),  \]
on which $J_K^S$ acts by multiplication through the homomorphism 
\[ \bar{\phi}_K^S \colon J_K^S \to G_K^S \]
induced from the restriction of $\bar{\phi}_K$ to the subgroup $\bA_{K,f}^* \cap \prod'_{\fp \not \in S}(K_\fp, \cO_\fp)$. For example, $G_K^\emptyset $ is isomorphic to the narrow class group $\mathit{Cl}_K^1:=J_K /P_K^1$ (where $P_K^1:= \{ (k) \in J_K \mid k \in K_+^*\}$) and $\phi _K^\emptyset \colon J_K \to \mathit{Cl} _K^1$ is equal to the quotient. 

Now, we use a canonical choice of a complement  
\[ J_{K,S}:=J_K^{S^c}= \bigoplus _{\fp \in S} \fp ^{\bZ} \]
of $J_{K}^S$. Then, $J_K \cong J_K^S \times J_{K,S}$ and $J_{K,S}$ acts on $X_\val^S$ freely. Therefore, the subspace
\[ X_K^S:=\val^{-1}_K(X_\val^S)= \Big( \prod _{\fp \in S} \cO_\fp^\natural \times \prod_{\fp \not \in S} \{ 0 \} \Big) \times _{\hat{\cO}_K^*} G_K^{\mathrm{ab}} \subset X_K \]
(afterwards $X_K^S$ is often identified with the space $\big( \prod _{\fp \in S} \cO_\fp^\natural  \big) \times _{\hat{\cO}_K^*} G_K^{\mathrm{ab}}$) is $J_K$-equivariantly isomorphic to $J_{K,S} \times G_K^S$ on which $J_K$ acts by multiplication through 
\[ \mathrm{id}_{J_{K,S}} \times \bar{\phi}_K^S \colon J_{K,S} \times J_K^S \to J_{K,S} \times G_K^S.\]
Indeed, the identification $\eta_K^{S} \colon J_{K,S} \times G_K^S \to X_K^S$ is explicitly given by
\begin{align}
 \eta_K^{S}(\fa ,\gamma):=[a_0^S \cdot a, \phi_K(a)^{-1}  \tilde{\gamma} ] \in \bA_{K,f}^* \times _{\hat{\cO}_K^*} G_K^{\mathrm{ab}}, \label{form:ident}
\end{align}
where $a \in \bA_{K,f}^*$ with $(a)=\fa$ and $\tilde{\gamma} \in G_K^{\mathrm{ab}}$ is a lift of $\gamma \in G_K^S$. Here we write $a_0^S$ for the finite adele determined by $(a_0^S)_\fp =1$ for $\fp \in S$ and $(a_0^S)_\fp =0$ for $\fp \not \in S$.

In summary, we get the following.
\begin{lem}\label{lem:XYdecomp}
There is a bijection
\[
\eta_K:= \Big( \bigsqcup_{S \subset \cP_K} \eta_K^S \Big)  \colon  \bigsqcup _{S\subset \cP_K } (J_{K,S} \times  G_K^S ) \to X_K.
\]
Moreover, it restricts to a bijection between $\bigsqcup _{S \subset \cP_K }I_{K,S} \times G_K^S$  and $Y_K$, where $I_{K,S}:=\bigoplus_{\fp \in S} \fp^\bN $. 
\end{lem}

Next we describe the topology on $Y_K$ and each $\overline{Y_K^S}$ in terms of the above decomposition.
For $S \subset S' \subset \cP_K$, let
\[ q_K^{S',S} \colon \Big(\prod_{\fp \in S'} \cO_\fp  \Big) \times _{\hat{\cO}_K^*}G_K^{\mathrm{ab}} \to \Big( \prod_{\fp \in S}\cO_\fp \Big) \times _{\hat{\cO}_K^*}G_K^{\mathrm{ab}}
\] 
denote the surjection induced from the projection $\prod_{\fp \in S'} \cO_\fp \to \prod_{\fp \in S}\cO_\fp $. 
Then, by (\ref{form:ident}), the composition
\[\theta_K^{S',S}:= (\eta_K^{S})^{-1} \circ q_K^{S',S} \circ \eta_K^{S'} \colon I_{K,S'} \times G_K^{S'} \to I_{K,S} \times G_K^{S} \]
is written as
\begin{align}
 \theta_K^{S',S}(\fa \fb, \gamma )= (\fa , \phi_K^{S}(\fb)^{-1} \cdot \pi_K^{S',S}(\gamma )) \label{form:theta}	
\end{align}
for $\fa \fb \in I_{K, S'} $, where $\fa \in I_{K,S}$ and $\fb \in I_{K,S'} \cap I_K^{S}$, and $\gamma \in G_{K}^{S}$. Here $\pi_K^{S',S} \colon G_K^{S'} \to G_K^{S}$ denotes the projection.

For $T \subset S \subset \cP_K$, we define the map $\Theta_T$ from $\bigsqcup _{S' \subset S} I_{K,S'} \times G_K^{S'}$ to the one-point compactification $(I_{K,T} \times G_K^T)^+=I_{K,T} \times G_K^T \cup \{ \ast \}$ by 
\[ \Theta_{T}|_{I_K \times G_K^{S'}} = \left\{  \begin{array}{ll}\theta_K^{S',T} & \text{ if $T \subset S'$, }\\ c_\ast & \text{ otherwise},  \end{array} \right. \]
where $c_\ast$ denote the constant map to $\ast$.

\begin{lem}\label{prp:action}
Let $S \subset \cP_K$. The bijection $\eta_K$ in Lemma \ref{lem:XYdecomp} gives rise to a homeomorphism between $\overline{Y_K^S}$ and $\bigsqcup_{S' \subset S} I_{K,S'} \times G_K^{S'}$ with the weakest topology such that $\Theta_T$ is continuous with respect to the usual topology of $I_{K,T} \times G_K^T$ for each $T \subset \cP_K$. 
\end{lem}
\begin{proof}
We write $Z_K^S$ for the set $\bigsqcup_{S' \subset S} I_{K,S'} \times G_K^{S'}$ with the topology given in the statement of the lemma. Note that $Z_K^S$ is Hausdorff because its topology is induced from the inclusion $\prod_T \Theta_T$.
The composition $\Theta_T \circ \eta_K^{-1} \colon \overline{Y_K^S} \to (I_{K,T} \times G_K^T)^+$ is continuous because it is the composition of $q_K^{S,T} \colon \overline{Y_K^S} \to \overline{Y_K^T} $ with the collapsing map $\overline{Y_K^T} \to (Y_K^T)^+$. That is, $\eta_K^{-1}$ gives a continuous bijection from $\overline{Y_K^S}$ to $Z_K^S$, which is actually a homeomorphism because $\overline{Y_K^S}$ is compact and $Z_K^S$ is Hausdorff.
\end{proof}

For a finite subset $F \subset \cP_K$, the structure of the $J_K$-space $X_K^F$, that is, the homomorphism $\phi_K^F \colon J_{K}^F \to G_K^F$ is well understood in class field theory.

\begin{lem}\label{lem:Galois}
Let $F$ be a finite subset of $\cP_K$.
\begin{enumerate}
\item The group $G_K^F$ is an extension of the narrow class group $\Cl_K^1$ by a quotient of the group $\prod _{\fp \in F}\cO_\fp^*$. 
\item The homomorphism $\phi_K^F$ factors through an isomorphism
\[\varprojlim \nolimits _{\mathfrak{m} \in I_{K,F}} J_K^F / P_K^{\mathfrak{m}} \to G_K^F, \]
where $I_{K,F} \cong \bN ^F$ is equipped with the product partial order and 
\[ P_K^{\mathfrak{m}}:=\{(k) \in J_K \mid \text{$k \in K_+^*$ such that $k \equiv 1$ modulo $\mathfrak{m}$} \}.\] 
\end{enumerate}
\end{lem}
\begin{proof}[Proof (cf.\ {\cite[Proposition 1.1]{MR3054305}})]
Let $K_\infty:=\prod_{\fp | \infty}K_\fp $ be the completion at all infinite places and let $K_\infty^o$ be the connected component of $K_\infty^*$.
Then, the Artin reciprocity map gives an isomorphism $\bA_K^*/\overline{K_\infty^oK^*} \to G_K^{\mathrm{ab}}$, where $\bA_K^*:=\bA_{K,f}^*\times K_\infty^*$.

We write as $\cO_\fp^{(0)}:=\cO_\fp^*$ and $\cO_\fp^{(n)}:=(1+\fp ^n)$ for $n\geq 1$.
For $\mathfrak{m} \in I_K$, let $U^{(\mathfrak{m})}_f:= \prod_{\fp} \cO_\fp^{(\mathfrak{m}_\fp)}$ and $U^{\mathfrak{m}}:=U^{\mathfrak{m}}_f \times K_\infty^o$. Then, since $U^{(\mathfrak{m})}$ is an open subgroup of $\bA_K^*$, we get
\[G_K^F \cong \bA_{K}^*/ \overline{K^*K_\infty^o}(\prod\nolimits_{\fp \not \in F} \cO_\fp^*) \cong \varprojlim \nolimits_{\mathfrak{m}} \bA_{K}^*/\overline{K^*K_\infty^o}U^{(\mathfrak{m})}_f \cong \varprojlim \nolimits_{\mathfrak{m}} \bA_K^* / K^* U^{(\mathfrak{m})}. \]
The right hand side is by definition isomorphic to the projective limit $\varprojlim \nolimits_{\mathfrak{m}} C_K/C_K^{\mathfrak{m}}$ (see \cite[Definition VI.1.2, Definition VI.1.7]{MR1697859}), which is isomorphic to $\varprojlim \nolimits _{\mathfrak{m}}J_K^F / P_K^{\mathfrak{m}}$ by \cite[Proposition VI.1.9]{MR1697859}. These isomorphisms are all $J_K$-equivariant by construction.

Now, (1) follows from (2) because $G_K^F /\phi_K(\prod _{\fp \in F}\cO_\fp^*) \cong G_K^\emptyset$ is isomorphic to $\Cl_K^1$. 
\end{proof}

Finally we get the following reconstruction of the Bost--Connes semigroup action $Y_K \curvearrowleft I_K$.

\begin{prop}\label{prp:global}
Let $K$ and $L$ be number fields. Let us fix a bijection $\chi \colon \cP_K \to \cP_L$ and write $j_\chi \colon J_K \to J_L$ and $j_\chi^F \colon J_K^F \to J_L^{\chi(F)}$ for the induced isomorphisms. Assume that there is a family of isomorphisms $\Phi ^F \colon G_K^F \to G_L^{\chi(F)}$ for any finite subset $F \subset \cP_K$ such that the diagrams
\begin{align}
\begin{gathered}
\xymatrix{
 J_K^F \ar[d]^{j_\chi^F} \ar[r]^{\phi_K^F} & G_K^F \ar[d]^{\Phi^F} \\
 J_L^{\chi(F)} \ar[r]^{\phi _L^{\chi(F)}} & G_L^{\chi(F)} \\
}\label{diag:JG}
\end{gathered}
\end{align}
commute. Then, there is a homeomorphism $\Psi \colon Y_K \to Y_L$ such that $(\Psi, j_\chi)$ gives rise to a conjugate of semigroup actions $Y_K \curvearrowleft I_K$ and $Y_L \curvearrowleft I_L$.
\end{prop}

\begin{proof}
In the proof, we omit $\chi$ and use the same symbol $F$ for its image in $\cP_L$ for simplicity of notation. First, consider the diagram
\begin{align}
\begin{gathered}
\xymatrix{
J_K^F \ar[r]^{\phi_K^F}  \ar[d]^{j_\chi^F} & G_K^F \ar[r]^{\pi_K^{F,E}} \ar[d]^{\Phi^F} & G_K^{E} \ar[d]^{\Phi^E} \\
J_L^F \ar[r]^{\phi_L^F} &G_L^F \ar[r]^{\pi_L^{F,E}} & G_L^E
}\label{diag:GG}
\end{gathered}
\end{align}
for finite subsets $E \subset F$ of $\cP_K$, where $\pi_K^{F,E}$ denotes the quotient $G_K^F \to G_K^E$. Then the left square commutes by assumption. The large outer square also commutes because it is a restriction of (\ref{diag:JG}) to subgroups $J_K^F$ and $J_L^F$. Since $\phi_K^F$ has a dense image, the right square also commutes.

Recall that $\theta_K^{F,E}$ is written in (\ref{form:theta}) by using only $\phi_K^F$ and $\pi_K^{F,E}$. Therefore, by the commutativity of (\ref{diag:JG}) and (\ref{diag:GG}), the diagram
\[\xymatrix@C=5em{
I_{K,F} \times G_K^{F} \ar[r]^{j_\chi^{F} \times \Phi^{F}} \ar[d]^{\theta_K^{F,E}} & I_{L,F} \times G_{L}^{F} \ar[d]^{\theta_L^{F,E}} \\ \ar[r]^{j_\chi^{E} \times \Phi^{E}} I_{K,E} \times G_{K}^{E}  & I_{L,E} \times G_L^{E}
}\]
commutes. Hence the map
\[ \Psi _F:= \bigsqcup _{E \subset F} j_\chi^{E} \times \Phi ^{E} \colon  \bigsqcup I_{K, E} \times G_K^{E} \to \bigsqcup I_{L,E} \times G_L^{E} \]
is a homeomorphism from $\overline{Y_K^F}$ to $\overline{Y_L^{F}}$ by Lemma \ref{prp:action}. Moreover, each $\Psi_F$ and $j_\chi$ gives rise to a conjugacy of $\overline{Y_K^{F}} \curvearrowleft I_K$ and $\overline{Y_L^E} \curvearrowleft I_L$.  

Now we get a homeomorphism $\Psi \colon Y_K \to Y_L$ as the projective limit of $\Psi_F$'s. Indeed, $Y_K$ is the projective limit $\varprojlim _F \overline{Y_K^F}$ by the connecting maps $\sigma_K^{F,E}:=\bigsqcup_{F' \subset F} \theta_K^{F', F' \cap E}$,
where $F$ runs over all finite subsets of $\cP_K$. Since each $\sigma_K^{F,E}$ is $I_K$-equivariant, the pair $(\Psi, j_\chi)$ gives rise to a conjugate of $Y_K \curvearrowleft I_K$ and $Y_L \curvearrowleft I_L$.
\end{proof}

\subsection{Subquotients of $A_K$}\label{section:2.3}
Here we observe how the decomposition given in Subsection \ref{section:2.2} is reflected to the structure of the Bost--Connes C*-algebra $A_K$. Throughout this paper, we use the symbol $\varphi \rtimes \Gamma \colon A \rtimes \Gamma \to B \rtimes\Gamma$ for the $\ast$-homomorphism between (reduced) crossed product C*-algebras induced from a $\Gamma$-equivariant $\ast$-homomorphism $\varphi \colon A \to B$.

As is proved in \cite[Proposition 3.17]{MR3545946}, there is a continuous surjection
\[ \psi_K \colon \Prim (A_K) \to X_{\val}/J_K  \cong 2^{\cP_K}, \]
that is, $A_K$ has a canonical structure of the C*-algebra over $2^{\cP_K}$ in the sense of \cite{MR1796912} (see also \cite[Definition 2.3]{MR2545613}). This map is characterized by the property that the pull-back of a $J_K$-invariant open subset $U$ of $X_\val$ corresponds to the ideal 
\[ A_K(U):= 1_{Y_K} (C_0(\val_K^{-1}(U)) \rtimes J_K ) 1_{Y_K} \]
of $A_K$. 

This $\psi_K$ is an intrinsic structure of the C*-algebra $A_K$ in the following sense. 
\begin{lem}\label{lem:prim}
Let $K$ and $L$ be number fields. Assume that there is an isomorphism $\varphi \colon A_K \to A_L$. Then, there is a bijection $\chi \colon \cP_K \to \cP_L$ such that the diagram
\[\xymatrix{\Prim (A_K) \ar[r]^{(\varphi^{-1})^*} \ar[d]^{\psi_K} & \Prim (A_L) \ar[d]^{\psi_L}\\ 2^{\cP_K} \ar[r]^{\chi } & 2^{\cP_L} }\]
commutes. That is, $\varphi \colon A_K \to A_L$ is a $\ast$-isomorphism over $2^{\cP_K}$.
\end{lem}
\begin{proof}
Let $\Prim _2(A)$ denote the set of second maximal primitive ideals  in the sense of \cite[Definition 3.9]{MR3545946}. 
It is proved in \cite[Proposition 3.11]{MR3545946} that $\Prim_2(A_K)  \subset \Prim(A)$ is a locally compact Hausdorff space and its connected components are in one-to-one correspondence with the elements of $\cP_K$. More precisely, there is a locally constant map
\[ \tilde{\psi}_K \colon \Prim_2(A_K) \to \cP_K \]
such that $\psi _K(P)=\{ \tilde{\psi}_K(P) \}^c \in 2^{\cP_K}$.
Therefore, we get a bijection
\[ \chi \colon  \cP_K \xrightarrow{\tilde{\psi}_K^{-1}} \pi_0(\Prim_2(A_K)))  \xrightarrow{(\varphi^{-1})^*} \pi_0(\Prim_2(A_L)) \xrightarrow{\tilde{\psi}_L} \cP_L. \]
Moreover, this choice of $\chi$ makes the above diagram commute because for any $P \in \Prim (A)$ we have
\[\psi_K(P)=\{ \tilde{\psi}_K (Q) \mid Q \in \Prim_2 (A), \ P \subset Q \}^c \in 2^{\cP_K},   \]
which follows from \cite[Proposition 3.6]{MR3545946}.
\end{proof}

Similarly, the C*-algebra 
\[A_\val := 1_{Y_\val }(C_0(X_\val) \rtimes J_K) 1_{Y_\val }\]
also has the structure of a C*-algebra over $2^{\cP_K}$ which is characterized by $A_\val (U)=1_{Y_\val }(C_0(U) \rtimes J_K ) 1_{Y_\val}$ for any $J_K$-invariant open subset $U \subset X_\val$. Note that it is isomorphic to the tensor product of infinite copies of the Toeplitz algebra $\mathcal{T}:=1_{\bN} (C_0(\tilde{\bZ}) \rtimes \bZ) 1_{\bN}$. Moreover, the $\ast$-homomorphism  
\[ \Val_K := \val_K^* \rtimes J_K \colon C_0(X_\val) \rtimes J_K \to C_0(X_K) \rtimes J_K \]
gives rise to a $\ast$-homomorphism from $A_\val$ to $A_K$, for which we use the same symbol $\Val_K$. It maps $A_\val(U)$ to $A_K(U)$ for any open subset $U \subset 2^{\cP_K}$, that is, it is a $\ast$-homomorphism over $2^{\cP_K}$.

Next we relate the structure of a C*-algebra over $2^{\cP_K}$ on $A_K$ with the decomposition in Lemma \ref{lem:XYdecomp}. 
Following the terminology in \cite{MR2545613}, we say that a subset of $2^{\cP_K}$ of the form $U \setminus V$, where $U,V$ are open subsets of $2^{\cP_K}$, is locally closed. For a locally closed subset $Z=U \setminus V$ of $2^{\cP_K}$, we associate a subquotient
\[ A_K(Z):= A_K(U) / A_K(U \cap V) \]
of $A_K$, which is independent of the choice of such $U$ and $V$ and has the structure of a C*-algebra over $Z$. In particular, for a locally closed subset $Z$ and an open subset $U$ of $2^{\cP_K}$, we get an exact sequence
\begin{align}0 \to A_K(Z \cap U) \to A_K(Z) \to A_K(Z \setminus U) \to 0. \label{form:loccl} \end{align}

Recall that $\overline{\{ S \}}=\{T \in 2^{\cP_K} \mid S \subset T \}$ for $S \subset \cP_K$. For a finite subset $F \subset \cP_K$, 
\[\{F^c \} = \Big( \bigcap_{E \subsetneq F} \overline{\{ E^c \}}^c \Big) \setminus \overline{\{ F^c\}}^c \]
is a locally closed subset. 
\begin{dfn}\label{dfn:B}
We define the C*-algebras
\begin{align*}
B_K^F&:= A_K(\{ F^c \} ) = 1_{Y_K^F} (C_0(X_K^F) \rtimes J_K) 1_{Y_K^F}, \\
B_\val^F &:= A_\val (\{ F^c \}) = 1_{Y_\val^F} (C_0(X_\val^F) \rtimes J_K) 1_{Y_\val^F}.
\end{align*}
\end{dfn}
The C*-algebras $B_K^F$ and $B_\val^F$ will play the role of composition factors of $A_K$ and $A_\val$ respectively. Note that $B_\val^F$ is canonically isomorphic to the tensor product of $C^*_rJ_K^F$ with the compact operator algebra $\bK(\ell^2(I_{K,F}))$. Moreover, since $\Val_K$ is a $\ast$-homomorphism over $2^{\cP_K}$, we get a $\ast$-homomorphism
\[ \Val_K^{F} \colon B_\val^F \to B_K^F.\]
\begin{lem}\label{lem:xi}
Let $\sB_K^F:=C(G_K^F) \rtimes J_K^F$. Then, there is an isomorphism
\begin{align*}
\xi_K^F \colon B_K^F \to \sB_K^F \otimes \bK(\ell^2(I_{K,F})),
\end{align*}
such that $\xi_K^F \circ \Val_K^{F}=(\pi_K^{F*} \rtimes J_K^F) \otimes \id$,
where $\pi_F \colon G_K^F \to \pt$ is the projection.
\end{lem}
\begin{proof}
The isomorphism $\xi_K^F$ is given by the restriction of 
\[ \eta_K^{F*} \rtimes J_K \colon C_0(X_K^F) \rtimes J_K^F \to C_0(J_{K,F} \times G_K^F) \rtimes J_K \cong (C(G_K^F)\rtimes J_K) \otimes \bK(\ell^2(J_{K,F}))
\] 
to the subalgebra $B_K^F$,  where $\eta_K^{F*} \colon J_K^F \times G_K^F \to X_K^F$ is the map given in (\ref{form:ident}). The second claim follows from $\val_K \circ \eta _K^{F*} =\id \times \pi_F$. 
\end{proof}

For a finite subset $F \subset \cP_K$ and $\fp \in F^c$, let $F_\fp:= F \cup \{ \fp\}$. Then, the two-point subset
\[ \{ F^c, F_\fp^c \} = \Big( \bigcap _{E \subsetneq F} \overline{\{ E_\fp^c \}}^c \Big) \setminus \overline{\{F_\fp^c\}}^c\]
is locally closed. We apply (\ref{form:loccl}) for $Z=\{ F^c , F_\fp^c \}$ and $U:=\overline{\{ F^c \} }^c$ to get exact sequences
\begin{align}
\begin{split}
0 \to B_K^{F_\fp } \to  A_K(\{ F^c, &F_\fp^c \}) \to B_K^F \to 0, \\
0 \to B_\val^{F_\fp } \to A_\val(\{ F^c, &F _\fp^c \}) \to B_\val^F \to 0. 
\end{split}
\label{form:exact}
\end{align}
Note that $A_K(\{ F^c, F_\fp^c \})$ and $A_\val (\{ F^c, F_\fp^c \})$ are explicitly written as
\begin{align*}
A_K(\{ F^c, F_\fp^c \}) & \cong 1 (C_0 (X_K^F \cup X_K^{F_\fp}) \rtimes J_K) 1,\\
A_\val(\{ F^c , F_\fp ^c\} ) & \cong 1 (C_0(X_\val^F \cup X_\val^{F_\fp}) \rtimes J_K) 1,
\end{align*}
where $X_K^F \cup X_K^{F_\fp}$ (resp.\ $X_\val^F \cup X_\val^{F_\fp}$) is equipped with the topology as an open subset of $\overline{X_K^{F_\fp}}$ (resp.\ $\overline{X_\val^{F_\fp}}$) and $1$ is the constant function on $\overline{Y_K^{F_\fp}}$ (resp.\ $\overline{Y_\val^{F_\fp}}$).
\begin{dfn}\label{dfn:D}
We write 
\begin{align*}
\partial_K^{F,\fp} &\colon \K_*(B_K^F) \to \K_{*+1}(B_K^{F_\fp}) ,\\
\partial_\val^{F,\fp} &\colon \K_*(B_\val^F) \to \K_{*+1}(B_\val^{F_\fp}) ,
\end{align*}
for the boundary homomorphism associated to the exact sequences (\ref{form:exact}).
\end{dfn}
\begin{rmk}\label{rmk:Toeplitz}
Since $X_{\val}^F \cup X_\val^{F_\fp}$ is identified with the subspace
\[\Big( \prod_{\fq \in F}\bZ \Big) \times {\bZ^+} \times \Big( \prod _{\fq \not \in F_\fp} \{ \infty \} \Big) \subset \Big(\prod _{\fq \in F} \bZ^+\Big) \times \bZ^+ \times \Big(\prod _{\fq \not \in F_\fp}\bZ^+ \Big) \cong X_\val ,\]
the second exact sequence in (\ref{form:exact}) is isomorphic to the tensor product of $C^*_rJ_K^{F_\fp} \otimes \bK(\ell^2(J_{K,F}))$ with the Toeplitz extension
\[ 0 \to \bK(\ell^2(\fp^{\bN })) \to \mathcal{T} \to C^*_r(\fp^\bZ) \to 0. \]
Therefore, $\partial_\val ^{F,\fp}$ is given by the Kasparov product with the $\KK_1$-class $[\cT] \in \KK_1(C^*(\fp^\bZ) , \bC)$ represented by the Toeplitz extension. (Recall that an extenson of C*-algebras determines an element of the $\KK_1$-group. A basic reference is \cite[Section 18]{MR1656031}.)
\end{rmk}

Finally, we discuss the use of $\KK(\fX)$-theory in the study of the Bost--Connes C*-algebra. Here we omit the detail of $\KK(\mathfrak{X})$-theory~\cite{MR1796912} (see also \cite{MR2545613} and \cite{Bentmann}) and only remark the following two points. First, for two C*-algebras $A$, $B$ over a topological space $\fX$, a $\ast$-homomorphism over $\fX$ from $A$ to $B$ gives an element of $\KK (\mathfrak{X}; A, B)$. Second, an element $\varphi \in \KK (\mathfrak{X}; A_K, A_L)$ induces a family of homomorphisms
\[\varphi_{Z*} \colon \K_*(A(Z)) \to \K_*(B(Z))\]
for any locally closed subsets $Z \subset \mathfrak{X}$ such that the diagrams
\[\xymatrix{
\K_*(A(Z\setminus W )) \ar[r] \ar[d]^{\varphi_{(Z\setminus W)*}} & \K_*(A(Z)) \ar[r] \ar[d]^{\varphi_{Z*}} & \K_*(A (W)) \ar[r]^\partial \ar[d]^{\varphi_{W*}} & \K_{*+1}(A(Z \setminus W)) \ar[d]^{\varphi_{Z*}} \\ 
\K_*(B(Z \setminus W)) \ar[r] & \K_*(B(Z)) \ar[r] & \K_*(B(W)) \ar[r]^\partial & \K_{*+1}((B(Z \setminus W))
}
\] 
commute for any closed subset $W \subset Z$ (\cite[Definition 2.4]{MR2545613}, see also \cite[Proposition 3.2.1]{Bentmann}). Note that a $\KK(\fX)$-equivalence given by a $\ast$-isomorphism over $\fX$ is ordered, that is, each of the induced isomorphism $\varphi_{Z*}$ gives a bijection between positive cones $\K_0(A(Z))_+ \cong \K_0(B(Z))_+$. 

We apply this commutativity for the exact sequences (\ref{form:exact}).  First, since $\Val_K$ is a $\ast$-homomorphism over $2^{\cP_K}$, the induced homomorphisms
\[\Val_{K*}^{F} \colon \K_*(B_\val ) \to \K_*(B_K^F) \]
make the diagrams
\begin{align}
\begin{gathered}
\xymatrix{
\K_*(B_\val^{F}) \ar[r]^{\partial _K^{F,\fp}} \ar[d]^{\Val^F_{K*}} & \K_{*+1}(B_\val^{F_\fp}) \ar[d]^{\Val^{F_\fp}_{K*}} \\
\K_*(B_K^{F}) \ar[r]^{\partial _\val^{F,\fp}} & \K_{*+1}(B_K^{F_\fp}) 
}\label{form:diagval}
\end{gathered}
\end{align}
commute. Second, if $A_L$ is regarded as a C*-algebra over $2^{\cP_K}$ by a fixed identification $\chi \colon \cP_K \to \cP_L$ and there is an $\KK(2^{\cP_K})$-equivalence $\varphi \in \KK(2^{\cP_K}; A_K, A_L)$, then it gives a family of isomorphisms
\[ \varphi_*^F \colon \K_*(B_K^F) \to \K_*(B_L^{\chi(F)})\]
such that the diagrams
\begin{align}
\begin{gathered}
\xymatrix@C=4em{
\K_*(B_K^{F}) \ar[r]^{\partial _K^{F,\fp}} \ar[d]^{\varphi^F} & \K_{*+1}(B_K^{F_\fp}) \ar[d]^{\varphi^{F_\fp}} \\
\K_*(B_L^{\chi(F)}) \ar[r]^{\partial _L^{\chi(F),\chi(\fp)}} & \K_{*+1}(B_L^{\chi(F_\fp)})} \label{form:diagphi}
\end{gathered}
\end{align}
commute for each finite subset $F \subset \cP_K$. Given the $\ast$-isomorphism $\varphi \colon A_K \to A_L$ and $\chi $ given in Lemma \ref{lem:prim}, we have such a $\KK(2^{\cP_K})$-equivalence. 
We remark that the discussion in this paragraph shows (4) $\Rightarrow$ (5) $\Rightarrow $ (6) of Theorem \ref{thm:main}.

\section{Reconstructing profinite completions from $\K$-theory of the crossed product C*-algebra}\label{section:3}
In this section we study C*-algebras of the form $C(G) \rtimes \Gamma$, where $\Gamma$ is a countable free abelian group and $G$ is its profinite completion. 
Our goal is to show that its $\K$-group remembers the completion, that is, the homomorphism $\Gamma \to G$. For simplicity of notation, we use the same symbol $\pi$ for quotients of compact abelian groups when the domain and range are specified. 

Throughout this paper, for a C*-algebra $A$ we use the symbol $\K_*(A)$ for the $\bZ /2$-graded group $\K_0(A) \oplus \K_1(A)$. The tensor product of $\K_*$-groups is also taken in the category of $\bZ /2$-graded abelian groups. 

Let $\cN$ be a set of rational prime numbers. A group $G$ is a \emph{pro-$\cN$ group} if it is a projective limit of finite groups whose orders are factorized as a product of primes in $\cN$. A \emph{pro-$\cN$ completion} of a discrete group $\Gamma$ is a pro-$\cN$ group $G$ equipped with a group homomorphism $f \colon \Gamma \to G$ whose image is dense in $G$. In other words, $G$ is a projective limit of finite quotients $\Gamma / \Gamma _n$ of $\Gamma$, where $\Gamma _n$ is a decreasing sequence of normal subgroups of $\Gamma$ such that $[\Gamma : \Gamma _n]$ is factorized as a product of primes in $\cN$.  We remark that we do not assume that the homomorphism $f$ is injective. For example, the quotient $f \colon \Gamma \to \Gamma /\Pi$ is a pro-$\cN$ completion if $[\Gamma : \Pi]$ is factorized as the product of primes in $\cN$.

Hereafter we deal with a finitely generated pro-$\cN$ completion of a countable free abelian group. We say that a profinite group $G$ is finitely generated if it is topologically finitely generated, that is, there is a finite family of elements of $G$ spanning a dense subgroup. For a set $\cN$ of rational prime numbers, $\bZ [\cN ^{-1}]$ denotes the smallest subring of $\bQ$ containing $p^{-1}$ for all $p \in \cN$ and $M[\cN ^{-1}] := M \otimes _{\bZ} \bZ[\cN^{-1}]$ for an abelian group $M$.

\begin{lem}\label{lem:proN}
Let $\cN $ be a set of rational prime numbers and let $G$ be a finitely generated pro-$\cN$ group. Then, $G$ is decomposed as the product $\prod _{p \in \cN} G_p$ such that each $G_p$ is isomorphic to $\bZ _p ^{d_p} \times F_p$ for $d_p \in \bZ_{>0}$ and a finite $p$-group $F$.
\end{lem}
\begin{proof}
We use the notation $\cN | n$ for $p | n$ for all $p \in \cN$. The group $G$ is canonically regarded as a finitely generated module over the ring
\[ \varprojlim_{\cN | n} \bZ / n \bZ \cong \prod \nolimits _{p \in \cN} \bZ_p.\]
Let $1_{\bZ_p}$ denote the unit of $\bZ_p$, which is an idempotent in $\prod \bZ_p$, and set $G_p:= 1_{\bZ _p}\cdot G$. Then $G\cong \prod G_p$ and each $G_p$ is a finitely generated $\bZ_p$-module. Hence we get the conclusion by the structure theorem for finitely generated modules over a PID (principal ideal domain).
\end{proof}
We represent the order of profinite groups by using supernatural numbers like $|G|=\prod p^{l_p}$, where $l_p=\log _p |F_p|$ if $d_p=0$ and $l_p=\infty$ if $d_p\geq 1$.

\begin{lem}\label{lem:Gpro}
Let $F$ be a finite subset of $\cP_K$. Then, there is a finite set $\cN_F$ of rational prime numbers such that $ \phi _K^F \colon J_K^F \to G_K^F$ is a pro-$\cN_F$ completion.
\end{lem}
\begin{proof}
For a prime $\fp \in \cP_K$ over a rational prime number $p$, the group $U_\fp^{(1)}$ of principal units of $K_\fp$ is a pro-$p$ group and the quotient $\cO_\fp^*/U_\fp^{(1)}$ is finite (see for example \cite[Proposition II.5.3]{MR1697859}). Let $\cN_\fp$ denote the union of $\{ p \} $ with the set of prime numbers dividing $|\cO_\fp^*/U_\fp^{(1)}|$. Then, $\cO_\fp^*$ is a pro-$\cN_\fp$ group. Let $\cN_F$ denote the union of $\bigcup _{\fp \in F} \cN_\fp$ with the set of prime numbers dividing $h_K^1:=|\Cl _K^1|$. Then  $G_K^F$ is a pro-$\cN_F$ group by Lemma \ref{lem:Galois} (1). Moreover, the map $\phi_K^F$ has dense image by Lemma \ref{lem:Galois} (2). 
\end{proof}

\subsection{$\K$-groups of $C(G) \rtimes \Gamma $}
For the calculation of $\K_*(C(G) \rtimes \Gamma)$, we start with the case that $G$ is finite, that is, $G \cong \Gamma /\Pi$ for a finite index subgroup $\Pi$. 

First of all, we review a special case of Green's imprimitivity theorem~\cite[Theorem 17]{MR0493349}.
Let $\sigma$ denote the regular representation of $\Gamma$ to $\ell ^2(\Gamma/\Pi )$. Recall that 
\begin{align*}
C(\Gamma /\Pi) \rtimes \Gamma &\cong \overline{\mathrm{span}} [(C(\Gamma/\Pi ) \otimes 1)\cdot (( \sigma \otimes \id )(C^*_r\Gamma))]  \\
&\subset \bK( \ell^2(\Gamma/\Pi)) \otimes  C^*_r\Gamma .
\end{align*}
We write $\kappa$ for this inclusion.
Let $p \in C(\Gamma /\Pi)$ be the support function on $0 \in \Gamma /\Pi$. Then the $\ast$-homomorphism
\[ j:=p \otimes \id _{C^*_r\Gamma }  \colon C^*_r\Gamma \to \bK(\ell^2(\Gamma /\Pi )) \otimes C^*_r\Gamma  \]
bijects the subalgebra $C^*_r\Pi$ onto the full corner $p(C(\Gamma /\Pi ) \rtimes \Gamma )p$. Consequently, $j_0:=j|_{C^*_r\Pi}$ induces the isomorphism $\K_{*} (C^*_r\Pi) \cong \K_*(C(\Gamma /\Pi) \rtimes \Gamma)$.

\begin{lem}\label{lem:finite}
Let $\Gamma$ and $\Pi$ be as above. Let $\iota \colon C^*_r\Pi \to C^*_r\Gamma$ denote the inclusion, let $\pi \colon \Gamma /\Pi \to  \mathrm{pt} $ denote the quotient and let $\pi^* \rtimes \Gamma \colon \bC \rtimes \Gamma \to C(\Gamma /\Pi) \rtimes \Gamma$ denote the induced $\ast$-homomorphism. Then, the composition
\[ \iota_* \circ (j_0) _*^{-1} \circ (\pi ^* \rtimes \Gamma )_* \colon \K_*( C^*_r\Gamma) \to \K_*(C^*_r\Gamma ) \]
is multiplication by $[\Gamma : \Pi]$.  In particular, $(\pi ^* \rtimes \Gamma )_*$ is injective.
\end{lem}
\begin{proof}
By definition $j_*$ gives a canonical identification of $\K_*(C^*_r\Gamma)$ with $\K_*(\bK(\ell^2(\Gamma /\Pi)) \otimes C^*_r\Gamma)$.
Since $j \circ \iota = \kappa \circ j_0$, we get
 \[ j_* \circ (\iota_* \circ (j_0) _*^{-1} \circ (\pi ^* \rtimes \Gamma )) = \kappa _* \circ  (\pi ^* \rtimes \Gamma )_* = (\sigma \otimes \id _{C^*_r\Gamma })_* .\]
Since $\sigma$ is homotopic to the trivial representation onto the $[\Gamma :\Pi]$-dimensional vector space $\ell^2(\Gamma /\Pi)$ (because $\hat{\Gamma}$ is connected), the right hand side is multiplication by $[\Gamma :\Pi]$.
\end{proof}

Next, we give a more explicit calculation of the $\K$-group.
It is well-known in topological $\K$-theory that there is a canonical isomorphism between $\K_*(C^*_r\Gamma) \cong \K^*(\hat{\Gamma})$ and the exterior algebra $\lwedge ^* \Gamma$ (actually, this is an isomorphism as Hopf algebras). 
Here, each element of $\Gamma $ is of odd degree, that is, $\K_0(C^*_r\Gamma) \cong \lwedge ^{\mathrm{even}} \Gamma$ and $\K_1(C^*_r\Gamma ) \cong \lwedge ^{\mathrm{odd}}\Gamma$.  
In the context of $\K$-theory of C*-algebras, this isomorphism is understood in terms of the Kasparov product in the following way. First, the $\K$-group of the C*-algebra of the free abelian group generated by a single element $v$ is
\[ \K_*(C^*_r(\bZ v)) \cong \bZ [1] \oplus \bZ \beta_v \cong \lwedge ^* (\beta _v), \]
where the Bott element $\beta _v \in \K_1(C^*_r(\bZ v))$ is represented by the unitary $u_v$. For an independent family of elements $v_1,\dots , v_k \in \Gamma$, the Kasparov product determines the element
\[ \textstyle{\beta _{v_1} \hotimes \dots \hotimes \beta_{v_k} \in \K_*(\bigotimes_i C^*_r(\bZ v_i)) \cong \K_*(C^*_r(\bigoplus_i \bZ v_i)).} \]
We use the same letter for its image in $\K_*(C^*_r\Gamma)$. By choosing a basis $\{ v_i \}_{i}$ of $\Gamma$, we get the homomorphism
\begin{align}
\lwedge^* (\beta_{v_1}, \beta _{v_2},\dots )\ni \beta_{v_{i_1}} \wedge \dots \wedge \beta_{v_{i_k}} \mapsto \beta _{v_{i_1}} \hotimes \dots \hotimes \beta _{v_{i_k}} \in  \K_*(C^*_r\Gamma), \label{form:Kunneth}
\end{align}
which is well-defined by the graded commutativity of the Kasparov product \cite[Theorem 5.6]{MR582160}. It is actually an isomorphism due to the K\"{u}nneth formula (recall that the K\"{u}nneth homomorphism $\K_*(A) \otimes \K_*(B) \to \K_*(A \otimes B)$ is nothing but the Kasparov product as above). 
We canonically identify the left hand side with $\lwedge ^*\Gamma$ by the correspondence $v \mapsto \beta _v$. For a rank $k$ free abelian group equipped with an orientation, the element 
\[ \beta _\Sigma := \beta _{v_1} \hotimes \dots \hotimes \beta _{v_k} \in \K_*(C^*_r\Sigma) \] 
is independent of the choice of an oriented basis $\{ v_i\} _i$ and generates $\lwedge ^{k}\Sigma \cong \bZ$. Hereafter, we use the same symbol $\beta _\Sigma$ for its image in $\K_*(C^*_r\Gamma )$ if $\Sigma$ is an oriented direct summand (that is, a direct summand with a fixed orientation) of $\Gamma$.

\begin{lem}\label{lem:iota}
Let $\Pi$ be a finite index subgroup of $\Gamma $. Through the isomorphism $\K_*(C^*\Gamma ) \cong \lwedge ^* \Gamma$, the homomorphism $\iota_* \colon \K_*(C^*_r \Pi) \to \K_*(C^*_r\Gamma)$ is identified with the inclusion $\lwedge ^*\Pi \to \lwedge ^* \Gamma$.
\end{lem}
\begin{proof}
When $\Gamma =\bZ v$ and $\Pi = n\bZ v$, the statement can be checked directly. Actually, $\iota_*([1])=[1]$ and $\iota_*(\beta _{nv})=\beta _{nv}=n\beta_v$. For general $\Gamma$ and $\Pi$, let us choose a basis $\{ v_i \}_i$ of $\Gamma$ such that $\Pi = \bigoplus _i n_i\bZ v_i$.
Now the claim follows from the functoriality of the K\"{u}nneth isomorphism.
\end{proof}

In particular, $\iota _*$ is an isomorphism after tensoring with $\bZ [\cN^{-1}]$ if $[\Gamma :\Pi]$ is factorized as primes in $\cN$. Together with Lemma \ref{lem:finite}, we can see that so is $(\pi ^* \rtimes \Gamma )_* \colon \K_*(C^*_r\Gamma) \to \K_*(C(\Gamma /\Pi) \rtimes \Gamma)$. 
 
Now we go back to the study of the K-group of $C(G) \rtimes \Gamma $ for general $G$.
\begin{lem}\label{lem:incl}
Let $\Gamma$ be a countable free abelian group, let $\varphi \colon \Gamma \to G$ be a pro-$\cN$ completion and let $\pi \colon G \to \mathrm{pt} $ denote the quotient. Then, 
\[ (\pi^* \rtimes \Gamma )_* \colon \K_*(C^*_r\Gamma) \to \K_*(C(G) \rtimes \Gamma )\]
is an isomorphism after tensoring with $\bZ[\cN^{-1}]$.
\end{lem}
\begin{proof}
Since $C(G) \rtimes \Gamma$ is isomorphic to the inductive limit $ \varinjlim _k C(\Gamma /\Gamma _k) \rtimes \Gamma$, it follows from the above observations.
\end{proof}

Lemma \ref{lem:incl} means that the $\K$-group of $C(G) \rtimes \Gamma$ determines an intermediate subgroup
\begin{align}
\lwedge^* \Gamma \subset \K_*(C(G)\rtimes \Gamma ) \subset \lwedge^* \Gamma [\cN^{-1}].\label{form:interm}
\end{align}
Here we simply write $\lwedge ^*\Gamma [\cN ^{-1}]$ for $(\lwedge^* \Gamma) [\cN^{-1}] = \lwedge^* (\Gamma [\cN^{-1}])$.
In the following subsections, we observe that this data has rich information, enough to reconstruct the pro-$\cN$ completion $f \colon \Gamma \to G$. Hereafter we often regard $\K_*(C(G) \rtimes \Gamma )$ as a subgroup of $\bQ \Gamma $ and omit the homomorphism $(\pi ^* \rtimes \Gamma)_*$ for simplicity of notations.

\begin{dfn}
Let $\Gamma$ be a countable free abelian group and let $f \colon \Gamma \to G$ be a profinite completion of $\Gamma $. For $k \in \bZ_{>0}$, we write
\[ \cK_\Gamma^k(G) := \K_*(C(G) \rtimes \Gamma ) \cap \lwedge^k \bQ\Gamma . \]
\end{dfn}

We remark that Lemma \ref{lem:finite} means that 
\begin{align}
\cK_\Gamma^n(\Gamma /\Pi) = \frac{1}{[\Gamma : \Pi]}\cdot  \lwedge ^n\Pi \label{form:incl}
\end{align}
as subgroups of $\lwedge ^n \bQ \Gamma$ because the composition
\[ \cK_\Gamma ^* (\Gamma /\Pi) \xrightarrow{\cong} \lwedge ^* \Pi \xrightarrow{\frac{1}{[\Gamma : \Pi ]} \iota_* } \lwedge ^* \bQ \Gamma  \]
restricts to the standard inclusion $\lwedge ^* \Gamma \to \lwedge ^* \bQ \Gamma $.
Therefore, for a general profinite completion $G=\varprojlim \Gamma /\Gamma _k$, we get
\begin{align}
 \cK_\Gamma^n(G) = \bigcup _{k=1}^\infty \Big( \frac{1}{[\Gamma :\Gamma _k]}
 \cdot \lwedge ^n \Gamma _k \Big). \label{form:KG}
\end{align}
In particular, we get a direct sum decomposition 
\[\K_*(C(G) \rtimes \Gamma) = \bigoplus _n \cK_\Gamma^n(G).\]

\subsection{Reconstructing pro-$p$ completions}
We start with the case that $\cN=\{ p \}$. 
Hereafter, we use the following symbol: for an element $x$ of an abelian group $M$, we define the supernatural number $\delta(x,M)=\prod _p p^{l_p(x,M)}$ to be
\[ l_p (x,M ) := \sup \{ l \in \bZ_{>0}  \mid x \in p^k M \} \in \bZ_{>0} \cup \{ \infty \} . \]

The following lemma is a direct consequence of Lemma \ref{lem:finite}.

\begin{lem}\label{lem:divisor}
Let $f \colon \Gamma \to G$ be a pro-$p$ completion of a countable free abelian group. Let $\Sigma$ be an oriented rank $d$ direct summand of $\Gamma$. 
Then
\[ \delta (\beta _\Sigma , \cK_{\Gamma}^d(G)) =|G/\overline{f(\Sigma)}|.\]
\end{lem}
\begin{proof}
Let $\{ \Gamma _k\}_k$ be a decreasing sequence of subgroups of $\Gamma$ such that $f$ factors through an isomorphism $\varprojlim \Gamma /\Gamma _k \to G$. By (\ref{form:KG}), we have
\[\delta (\beta _\Sigma , \cK_\Gamma ^d(G))=\sup_k \delta (\beta _\Sigma , \cK _\Gamma ^d (\Gamma /\Gamma _k) ). \]
On the other hand, $|G/\overline{f(\Sigma)}|$ is equal to the supremum of $|\Gamma /(\Gamma _k+\Sigma)|$.  Hence the proof of the lemma is reduced to the case that $G=\Gamma /\Pi$.

Let us choose an oriented basis $v_1,\dots, v_d$ of $\Sigma$ such that $\Sigma \cap \Pi=\bigoplus _i n_i\bZ v_i$. We remark that $\Sigma \cap \Pi$ is a direct summand of $\Pi$. By Lemma \ref{lem:finite}, $\iota _* \circ (j_0)_*^{-1} \circ (\pi _* \rtimes \Gamma )_*(\beta _\Sigma )$ coincides with
\begin{align*}
[\Gamma : \Pi] \beta _{v_1} \wedge \dots \wedge \beta _{v_d} =& \frac{[\Gamma : \Pi ]}{n_1\dots n_d} \beta _{n_1v_1} \wedge \dots \wedge \beta _{n_dv_d}\\
=&\frac{[\Gamma : \Pi ]}{n_1\dots n_d}\iota _*(\beta _{\Sigma \cap \Pi })
\end{align*}
and hence 
\[ (j_0)_*^{-1}\circ (\pi_* \rtimes \Gamma )(\beta _\Sigma)=\frac{[\Gamma :\Pi]}{n_1\dots n_d}\beta _{\Sigma \cap \Pi}.\]
Since $\delta (\beta _{\Sigma \cap \Pi},\lwedge^d \Pi)=1$, we get   
\begin{align*}
\delta (\beta _\Sigma , \cK_\Gamma^d({\Gamma /\Pi })) &= \delta((j_0)_*^{-1}\circ (\pi_* \rtimes \Gamma )(\beta _\Sigma) , \lwedge ^* \Pi ) \\
&=\frac{[\Gamma : \Pi ]}{n_1\dots n_d} =\frac{[\Gamma : \Pi]}{[\Sigma : \Sigma \cap \Pi]} = |\Gamma /(\Pi +\Sigma)|. \qedhere
\end{align*}
\end{proof}

\begin{lem}\label{lem:rank1}
Let $\Gamma$ be a free abelian group and let $f \colon \Gamma \to G$ be a pro-$p$ completion such that $G \cong \bZ_p$. Then, $f$ factors through the isomorphism 
\[ \varprojlim _{k \to \infty } \Gamma / (\Gamma \cap p^k\cK_\Gamma^1(G)) \to G. \]
\end{lem}
\begin{proof}
Since 
\[G \cong \varprojlim_k G/p^kG \cong \varprojlim_k \Gamma /f^{-1}(p^kG), \]
it suffices to show that $\Gamma \cap \cK_\Gamma^1(G) = f^{-1}(p^kG)$.

We show that $x \in p^k\cK_\Gamma^1(G)$ if and only if $f(x) \in p^kG$ for $x \in \Gamma$. Without loss of generality we may assume $\delta (x,\Gamma )=1$ because $\delta (p^kx , \cK^1_\Gamma (G))=p^k \delta (x, \cK^1_\Gamma (G))$. 
Lemma \ref{lem:divisor} implies that 
\[ \delta (\beta _x, \cK_\Gamma^1(G))=|G/\overline{f(\bZ x)}|.\]
Recall that a closed subgroup $\overline{f(\bZ x)}$ of $G \cong \bZ _p$ is of the form $p^l G$ for some $l \geq 0$.
Now we get the conclusion because $p^k$ divides $|G/\overline{f(\bZ x)}|$ if and only if $f(x) \in p^k G$. 
\end{proof}

For $d \in \bZ_{>0}$, let $S_d(\Gamma ,G)$ denote the set of oriented rank $d$ direct summands $\Sigma$ of $\Gamma$ such that there exists $x \in \Gamma \setminus \Sigma$ with $\delta (\beta _x \wedge \beta _\Sigma ,\cK_\Gamma^{d+1}(G))=1$. Note that such $x$ satisfies $\delta (x,\Gamma)=1$ and hence $\bZ x \oplus \Sigma$ is also a direct summand of $\Gamma$. By the fundamental theorem of finitely generated modules over a PID and Lemma \ref{lem:divisor}, an oriented direct summand $\Sigma$ is in $S_d(\Gamma, G)$ if and only if $G/\overline{f(\Sigma)}$ is a singly generated $\bZ_p$-module.

For an oriented rank $d$ direct summand $\Sigma$ of $\Gamma$, let $\beta _\Sigma \wedge {} \cdot {}$ denote the endomorphism on $\lwedge ^* \Gamma [p^{-1}] $ taking the exterior product with $\beta _\Sigma$. In particular, it induces a homomorphism from $\Gamma [p^{-1}] \cong \lwedge ^1  \Gamma [p^{-1}]$ to $\lwedge^{d+1}  \Gamma [p^{-1}]$.
\begin{lem}\label{lem:rankn}
Let $\Gamma$ be a free abelian group, $G := \bZ_p^d $ and let $f \colon \Gamma \to G$ be a pro-$p$ completion. Set
\[ \Gamma _{k} := \Gamma \cap \bigcap _{\Sigma \in S_{d-1}(\Gamma ,G)}  (\beta _\Sigma \wedge {}\cdot {})^{-1}(p^k \cK_\Gamma^d(G)) . \]
Then, $f$ factors through the isomorphism $\varprojlim \Gamma /\Gamma _k \cong G$.
\end{lem}
\begin{proof}
For $\Sigma \in S_{d-1}(\Gamma , G)$, let $G_\Sigma:=G/\overline{f(\Sigma)}$
and let $f _\Sigma$ denote the composition of $f$ with the quotient $G \to G_\Sigma$. Note that $G_\Sigma$ is isomorphic to $\bZ_p$ since the $\bZ_p$-rank of $G_\Sigma$ is equal to $1$ and $G_\Sigma$ is singly generated. 

We claim that 
\[  (\beta _\Sigma \wedge {}\cdot{})^{-1}(\cK_{\Gamma}^d(G)) =  \cK_\Gamma ^1({G_\Sigma }). \]
Indeed, Lemma \ref{lem:divisor} implies that
\[\delta (\beta _x \wedge \beta _\Sigma ,\cK_\Gamma^d(G))= |G/\overline{f(\Sigma \oplus \bZ x)}|=|G_\Sigma/ \overline{f(\bZ x)})|=\delta (\beta_x, \cK_\Gamma^1({G_\Sigma }))\]
for $x \in \Gamma \setminus \Sigma $ with $\delta (x,\Gamma)=1$ and 
$\delta (\beta _x, \cK_\Gamma^1(G_\Sigma))=p^\infty $ for $x \in \Sigma $.

Now  $f_\Sigma $ factors through the isomorphism 
\[ \varprojlim _k \frac{\Gamma}{\Gamma \cap (\beta _\Sigma \wedge {} \cdot {})^{-1}(p^k\cK_\Gamma^d(G))} \to G _\Sigma \]
by Lemma \ref{lem:rank1}. Hence the direct product 
\[  \prod  f_\Sigma \colon \Gamma \to \prod _{\Sigma \in S_{d-1}(\Gamma ,G)}   G_\Sigma\]
 factors through an injection $\varprojlim \Gamma /\Gamma _k \to \prod G_\Sigma$.

For the proof of the lemma, it suffices to show that the product of projections $G \to \prod G_\Sigma$ is injective. Since $f$ has dense image, there is a finite family $\{ v_1,\dots ,v_d\}$ of elements of $\Gamma $ such that $\{ f(v_1), \dots ,  f(v_d)\}$ forms a free basis of the $\bZ_p$-module $G$. Now, $\Sigma _i:=\bigoplus _{j \neq i}\bZ v_j$ for $i=1,\dots ,d$ satisfies $\Sigma _i \in S_{d-1}(\Gamma , G)$ and $G \to \prod_i G_{\Sigma _i}$ is injective.
\end{proof}

\begin{lem}\label{lem:prepare}
Let $\Gamma$ be a free abelian group and let $f \colon \Gamma \to G$ be a pro-$p$ completion such that $G$ is finitely generated. Let $F$ denote the torsion subgroups of $G$.
\begin{enumerate}
\item Let $d$ be the minimal integer such that $\cK_\Gamma^d(G)$ is a proper subgroup of $\lwedge ^d \Gamma [p^{-1}]$. Then, it is equal to the rank of $G$.  
\item Let $\Lambda $ be an oriented rank $d$ direct summand of $\Gamma$ minimizing $\delta (\beta _\Lambda , \cK_\Gamma^d(G))$ and set $H:=\overline{f(\Lambda )}$. Then, $G/H$ is isomorphic to $F$. In particular, $N:=\delta (\beta _\Lambda , \cK_\Gamma^d(G))$ is equal to $|F|$.
\item The subgroup
\[ \Pi := \{ x\in \Gamma \mid N^{-1}\beta _x \wedge \beta _\Lambda \in \cK_\Gamma ^{d+1}(G) \} \]
is equal to $f^{-1}(H)$. 
\item Under the identification $\lwedge ^* \bQ \Pi \cong \lwedge ^* \bQ \Gamma $ induced from the inclusion $\Pi \to \Gamma$, we get
\[ \cK_\Pi^*(H) = N \cdot \cK_\Gamma ^*(G) .\]
Consequently, $S_d(\Pi,H)$ consists of oriented direct summands $\Sigma$ of $\Pi$ such that there exists $x \in \Pi$ with $\delta (\iota_*(\beta _x \wedge \beta _\Sigma), \cK_\Gamma^d(G))=N$.
\end{enumerate}
\end{lem}
\begin{proof}
Since $f$ has dense image, there is a finite family $\{ v_i \}_{i \in I}$ of elements of $ \Gamma$ such that $\{ f(v_i) \}_{i \in I}$ freely generates a submodule $M$ of $G$ such that $G\cong M \times F$. 
Now we get $\delta (\beta _\Lambda ,\cK_\Gamma^d (G)) <\infty $ by Lemma \ref{lem:divisor} and hence $\cK^d_\Gamma (G) \subsetneq \lwedge ^d [p^{-1}]$. On the other hand, Lemma \ref{lem:divisor} also implies that $\delta (\beta _\Sigma, \cK_\Gamma ^*(G) )=p^\infty$ for any direct summand $\Sigma$ of the rank less than $d$. 

In fact, this choice of $\Lambda $ actually minimizes $\delta (\beta _\Lambda , \cK_\Gamma ^d(G))$ as in the statement of (2). This follows from the fact that the order of the quotient of $G$ by any free subgroup of the same rank divides $|F|$. Now (2) follows from Lemma \ref{lem:divisor}.

Next we show (3). Let $\pi \colon G \to G/H$ denote the quotient. It is obvious that both $\Pi$ and $f^{-1}(H)$ contain $\Lambda$. On the other hand, for $x \in \Gamma \setminus \Lambda $ with $x=p^ly$ and $\delta(y,\Gamma)=1$, we have
\begin{align*}
\delta(\beta_x \wedge \beta_\Lambda, \cK^{d+1}_\Gamma (G))&=p^l \delta (\beta_y \wedge \beta_\Lambda, \cK^{d+1}_\Gamma (G))\\
&=p^l|G/\overline{f(\bZ y \oplus \Lambda)}|=p^l|F/\langle \pi \circ f(y) \rangle |
\end{align*}
by Lemma \ref{lem:divisor}. The right hand side is equal to $|F|=N$ if and only if the order of $f(y)$ divides $p^l$, that is, $f(x)=0$.

Finally, (4) immediately follows from (\ref{form:KG}) as
\begin{align*}
\cK_\Gamma^*(G) &= \bigcup _{k=1}^\infty \Big( \frac{1}{[\Gamma :\Gamma _k]} \cdot \lwedge ^* \Gamma _k \Big) \\
&= \frac{1}{[\Gamma : \Pi]}\bigcup _{k=1}^\infty \Big( \frac{1}{[\Pi :\Gamma _k]} \cdot \lwedge ^n \Gamma _k \Big)\\
&=N^{-1} \cK_\Pi^{*}(H),
\end{align*}
if we choose $\{ \Gamma _k \}$ as $\Gamma _1=\Pi$.
\end{proof}

\begin{thm}
Let $f \colon \Gamma \to G$, $d$, $\Lambda $, $N$ and $\Pi$ be as in Lemma \ref{lem:prepare}. 
Set
\[ \Pi _k:= \Pi \cap \bigcap _{\Sigma \in S_{d-1}(\Pi , H)} (\beta _\Sigma \wedge {}\cdot{})^{-1}(p^kN\cdot  \cK_\Gamma^d(G)) .\]
Then, $f$ factors through an isomorphism $\varprojlim \Gamma /\Pi _k \to G$.
\end{thm}
\begin{proof}
By construction, $H$ is a free $\bZ_p$-module. Therefore, $f \colon \Pi \to H$ factors through the isomorphism $\varprojlim _k \Pi/\Pi _k \to H$ by Lemma \ref{lem:rankn} and Lemma \ref{lem:prepare} (4). The result follows  because $G=\Gamma \times _\Pi H$.
\end{proof}

\begin{cor}\label{cor:pisom}
Let $\Gamma$ be a free abelian group, let $f_i\colon \Gamma \to G_i$ ($i=1,2$) be two pro-$p$ completions such that $G_i$ are finitely generated and let $\pi_i \colon G_i \to \mathrm{pt} $ denote the quotient. Suppose that we have the isomorphism
\[ \varphi \colon \K_*(C(G_1) \rtimes \Gamma ) \to \K_*(C(G_2) \rtimes \Gamma )\]
 such that $\varphi \circ (\pi_1^* \rtimes \Gamma )_*=(\pi_2^* \rtimes \Gamma )_* $. Then, there is a group isomorphism $\Phi \colon G_1 \to G_2$ such that $\Phi \circ f_1=f_2$.
\end{cor}
\begin{proof}
The assumption means that $\cK^*_\Gamma (G_1)=\cK^*_\Gamma (G_2)$ as intermediate subgroups of $\lwedge ^*\Gamma \subset \lwedge ^* \bQ \Gamma$.
It is checked in Lemma \ref{lem:prepare} that all the data $d$, $\Lambda$, $N$, $\Pi$ and $S_{d-1}(\Pi,H)$ used to define the decreasing sequence $\{ \Pi_k \}$ depend only on the intermediate subgroup $\lwedge ^* \Gamma \subset \cK^*_\Gamma (G) \subset \lwedge \bQ\Gamma$. Therefore we get isomorphisms $\hat{f}_i \colon \varprojlim _k \Gamma /\Pi_k \to G_i$ and $\Phi:= \hat{f}_2\circ \hat{f}_1^{-1}$ is the desired isomorphism.
\end{proof}

\subsection{Reconstructing pro-$\cN$ completions}
Let $\cN$ be a subset of prime numbers and let $f \colon \Gamma \to G$ be a pro-$\cN$ completion such that $G$ is finitely generated. By Lemma \ref{lem:proN}, the group $G$ is decomposed as $G_{p_1} \times \dots \times G_{p_k}$. For $p \in \cN$, let $\pi_p \colon G \to G_{p}$ be the projection. 

\begin{lem}\label{lem:decompKG}
For $p \in \cN$, the subgroup $\cK_\Gamma ^*(G) \cap \lwedge^* \Gamma [1/p]$ of $\lwedge^* \bQ \Gamma $ is equal to $\cK_\Gamma ^*(G_p)$.
\end{lem}
\begin{proof}
Let $\cN _p:= \cN \setminus \{ p \} $. Since $\cK_\Gamma ^*(G_p )$ is included in $\cK_\Gamma ^*(G) \cap \Gamma [1/p] $, it suffices to show that the inclusion $\cK _\Gamma ^*(G_p) \subset \cK_\Gamma ^*(G)$ induces an isomorphism after tensoring with $\bZ [\cN_p^{-1}]$. Moreover, by (\ref{form:KG}) it is enough to consider the case that $G=\Gamma /\Pi$. Let $\Pi_p$ be the subgroup of $\Gamma$ such that $\Gamma /\Pi_p \cong G_p$, that is, $[\Gamma : \Pi_p]$ is a power of $p$ and $p$ does not divide $[\Pi_p : \Pi ]$. Then  $\Pi _p [\cN_p^{-1}] = \Pi [\cN_p^{-1}]$ and hence
\begin{align*}
\cK_\Gamma^*(G) [\cN_p^{-1}]&= \Big( \frac{1}{[\Gamma : \Pi]}\lwedge ^* \Pi \Big) [\cN_p^{-1}]\\
& = \frac{1}{[\Gamma : \Pi _p ]}\lwedge ^* \Pi _p [\cN_p^{-1}] \\& = \cK_\Gamma ^* (G_p)[\cN_p^{-1}].\qedhere
\end{align*}
\end{proof}

\begin{cor}\label{cor:isom}
Let $f_i \colon \Gamma_i \to G_i$ ($i=1,2$) be two pro-$\cN$ completions of free abelian groups and let $\pi_i \colon G_i \to  \mathrm{pt} $ denote the quotient. Suppose that we have an isomorphism $F \colon \Gamma _1 \to \Gamma_2$ and \[ \varphi \colon \K_*(C(G_1) \rtimes \Gamma ) \to \K_*(C(G_2) \rtimes \Gamma )\]
 such that the diagram
\[\xymatrix@C=4em{
\K_*(C^*_r\Gamma_1) \ar[r]^{(\pi_1^* \rtimes \Gamma_1)_*} \ar[d]^{F_*} & \K_*(C(G_1) \rtimes \Gamma_1) \ar[d]^\varphi \\
\K_*(C^*_r\Gamma_2) \ar[r]^{(\pi_2^* \rtimes \Gamma_2)_*} & \K_*(C(G_2) \rtimes \Gamma _2)
}\]
commutes. Then, there is a group isomorphism $\Phi \colon G_1 \to G_2$ such that $\Phi \circ f_1=f_2 \circ F$.
\end{cor}
\begin{proof}
First, by replacing $f_2$ with $f_2 \circ F$, we may assume that $\Gamma_1=\Gamma_2$ and $F=\id$ without loss of generality.
For $p \in \cN$ and $i=1,2$, let $\pi_{i,p} \colon G_i \to (G_i)_p$ denote the quotient. By Lemma \ref{lem:decompKG}, $\varphi$ induces an isomorphism $\varphi _p \colon \cK_\Gamma ^* (G_{1,p}) \cong \cK_\Gamma ^*(G_{2,p})$ such that $(\pi _{1,p} ^*\rtimes \Gamma )_* \circ \varphi _p = (\pi_{2,p}^* \rtimes \Gamma )_*$. We apply Corollary \ref{cor:pisom} to get an isomorphism $\Phi _p \colon G_{1,p} \to G_{2,p}$. Finally, $\Phi:=\prod _{p \in \cN} \Phi _p $ is the desired isomorphism. 
\end{proof}

\section{Reconstructing the Bost--Connes semigroup action}\label{section:4}
In this section we give a proof of Theorem \ref{thm:main}. Throughout this section, we fix a total order on $\cP_K$ in order to fix the orientation on direct summands $J_{K,F}$ for finite subsets $F \subset \cP_K$. 
For finite subsets $F$ and $F'$ of $\cP_K$ with $F \cap F' = \emptyset$, we use the symbol $\beta _{F'}^F$ for the element $\beta_{J_{K,F'}} \in \K_0(C^*_rJ_{K}^F)$. We write $\pi_K^F$ for the projection $G_K^F \to \pt$. 

The essential step is (6)$\Rightarrow$(1). Here we reconstruct the semigroup action $Y_K \curvearrowleft I_K$ from the family of ordered groups $\K_*(B_K^F)$ and homomorphisms $\partial_{K}^{F,\fp} \colon \K_*(B_K^F) \to \K_{*+1}(B_K^{F_\fp})$. 
First of all, recall that we have an ordered isomorphism
\[ \xi_{K*}^F \colon \K_*(B_K^F) \to \K_*(\sB_K^F) ,\]
where $\xi_K^F$ is as in Lemma \ref{lem:xi}. 
We will apply Corollary \ref{cor:isom} to reconstruct profinite completions $\phi _K^F \colon J_K^F \to G_K^F$. To this end, we need to reconstruct the inclusion $\lwedge ^* J_K^F \to \K_*(\sB_K^F)$ from boundary homomorphisms. Recall that $\sB_K^\emptyset = C(J_K/P_K^1) \rtimes J_K$ and $|J_K/P_K^1|=:h_K^1$ is called the narrow class number.

\begin{lem}\label{lem:bdry}
The boundary homomorphism $\partial _K ^{F,\fp}$ is uniquely determined by the equalities
\begin{align*}
\partial _K^{F,\fp} (\Val_{K*}^F(\beta _{F'}^F)) &= 0,  \\
\partial _K^{F,\fp} (\Val_{K*}^F(\beta _{F'_\fp}^F)) &= (-1)^{N(F',\fp)+1} \Val_{K*}^{F_\fp}(\beta _{F'}^{F_\fp}),
\end{align*}
for any finite subset $F' \subset F_\fp^c $. Here, $N(F',\fp)$ denotes the inversion number $| \{ \fq \in F' \mid \fp < \fq \} |$. 
\end{lem}
\begin{proof}
By (\ref{form:interm}), a homomorphism from $\K_*(\sB_K^F)$ to a torsion-free group is uniquely determined by the image of its subgroup $(\pi_K^{F*} \rtimes J_K^F)_*(\K_*(C^*J_K^F))$. 
This shows the uniqueness of homomorphisms from $\K_*(B_K^F)$ to $\K_{*+1}(B_K^{F_\fp})$ with the above equalities since $\Val_{K*}^F = (\xi_{K*}^F) ^{-1} \circ (\pi_K^{F*} \rtimes J_K^F)_*$, as is shown in Lemma \ref{lem:xi}. 
 
Since the diagram (\ref{form:diagval}) commutes, it suffices to show 
\[ \partial _\val ^{F,\fp} (\beta _{F'}^F) = 0,  \ \ \  \partial _\val^{F,\fp} (\beta _{F'_\fp}^F) =(-1)^{N(F',\fp)+1} \beta _{F'}^{F_\fp}.\]
To see this, recall that the identification $\K_*(C^*_rJ_K) \cong \lwedge ^* J_K$ is given by the Kasparov product as in (\ref{form:Kunneth}). In particular, we have 
\[ \beta_{F'}^F=\beta_{F'}^{F_\fp} \hotimes [1_{C^*(\fp^\bZ)}], \ \ \   \beta_{F'_\fp}^{F} =(-1)^{N(F',\fp)} \beta _{F'}^{F_\fp} \hotimes \beta_\fp \]
as elements of $\K_*(C^*_rJ_K^F) \cong \K_*(C^*_rJ_K^{F_\fp} \otimes C^*_r(\fp^\bZ))$.  Now, by Remark \ref{rmk:Toeplitz} and the associativity of the Kasparov product, we get
\begin{align*}
\partial _{\val }^{F,\fp}(\beta _{F'}^F)&=(\beta_{F'}^{F_\fp} \hotimes [1_{C^*_r(\fp ^\bZ )}]) \hotimes _{C^*_r(\fp ^\bZ )} [\mathcal{T}] =0,  \\
\partial _{\val }^{F,\fp}(\beta _{F'}^F \wedge \beta _\fp )&=(-1)^{N(F',\fp)}(\beta_{F'}^{F_\fp} \hotimes \beta_{\fp}) \hotimes _{C^*_r(\fp ^\bZ )} [\mathcal{T}] =(-1)^{N(F',\fp)+1}\beta_{F'}^{F_\fp} 
\end{align*}
since $[1_{C^*_r(\fp ^ \bZ)}] \hotimes _{C^*_r(\fp ^ \bZ)} [\mathcal{T}]=0$ and $\beta_\fp \hotimes _{C^*_r(\fp ^ \bZ)} [\mathcal{T}]=-1 \in \KK(\bC,\bC)$.
\end{proof}

Let $F=\{ \fp_1,\dots, \fp _l \} $ be a finite subset of $\cP_K$ with $\fp _1 \leq \dots \leq \fp_l$. We write as $F_i:= \{ \fp _{l-i+1},\dots, \fp_l \} \subset F$ and 
\[  D_K^F:= \xi_K^F \circ \partial _K^{F_{l-1} ,\fp _1 }\circ \cdots \circ \partial _K^{F_1, \fp _{l-1} } \circ \partial _K^{\emptyset ,\fp _l} \circ (\xi_K^\emptyset )^{-1}  \colon \K_*(\sB_K^\emptyset ) \to \K_{*+l}(\sB_K^F).  \]
Similarly we define $D_\val^F \colon \K_*(C^*_rJ_K) \to \K_{*+l}(C^*_rJ_K^F)$. 

\begin{lem}\label{lem:tau}
There is a unique ordered homomorphism
\[ \tau _K^F \colon \K_0(\sB_K^F) \to \bR \]
such that the image $(\tau_K^F \circ D_K^F )(\K_l( \sB_K^\emptyset ))$ is $\bZ  $. Moreover, it maps $[1_{\sB_K^F}]$ to the narrow class number $h_K^1$.
\end{lem}
\begin{proof}
For a free abelian group $\Gamma$ and its finite index subgroup $\Pi$, the $\K_0$-group of the C*-algebra $C(\Gamma / \Pi) \rtimes \Gamma \cong C^*_r(\Pi) \otimes \bM_{|\Gamma /\Pi|}$ admits a unique ordered homomorphism to $\bR$ up to scalar multiplication (see \cite[Exercise 6.10.3]{MR1656031}). 
Hence the inductive limit
\[ \K_0(\sB_K^F)=\varinjlim_{\mathfrak{m} \in I_{K,F}} \K_0(C(J_K^F/P_K^\mathfrak{m}) \rtimes J_K^F)\]
(we remark that this equality is obtained from Lemma \ref{lem:Galois} (2)) also admits a unique ordered homomorphism to $\bR$ up to scalar multiplications. 

We take the unique ordered homomorphism $\tau_K^F$ such that $\tau_K^F([1_{\sB_K^F}])=h_K^1$. 
Then, the composition $\tau _K^F \circ (\pi _K^{F *} \rtimes J_K)_*$ has the image $h_K^1 \cdot \bZ$. This is because it coincides with the map induced from the canonical trace on $\K_*(C^*_rJ_K)$ multiplied with $h_K^1$ by the uniqueness of ordered homomorphisms from $\K_*(C^*_rJ_K^F) $ to $\bR$ mapping $[1_{C^*_rJ_K^F}]$ to $h_K^1$.  
Hence we obtain 
\begin{align*}
 h_K^1  (\tau_K^F \circ D_K^F )(\K_l(\sB_K^\emptyset )) & \subset (\tau_K^F \circ D_K^F)((\pi _K^{\emptyset *} \rtimes J_K)_* (\K_l(\sB_\val^\emptyset ))) \\ 
 & = (\tau_K^F \circ (\pi _K^{F *} \rtimes J_K^F)_*)(D_\val^F (\K_l(\sB_\val^\emptyset))) \\
 & \subset (\tau_K^F \circ (\pi _K^{F *} \rtimes J_K^F)_*) (\K_{0}(\sB_\val^F)) \subset  h_K^1 \cdot \bZ.
 \end{align*}
 Here we use the fact $h_K^1 \K_l(\sB_K^\emptyset ) \subset (\pi _K^{\emptyset *} \rtimes J_K)_* (\K_l(C^*_rJ_K ))$, which follows from (\ref{form:incl}).

The remaining task is to show that the image of $\tau_K^F \circ D_K^F$ contains $1 \in \bZ$. Recall that $\phi_K^\emptyset$ surjects $J_K^F$ to $\Cl_K^1$ by Lemma \ref{lem:Galois} (2). Therefore, for each $\fp_i \in F$ there is $\fa _i \in J_K^F$ such that $\fa _i \cdot \fp_i \in P_K^1$. Let $\iota$ and $j_0$ be $\ast$-homomorphisms as in Lemma \ref{lem:finite} for the inclusion $P_K^1 \subset J_K$ (note that $\iota = \Val_K^\emptyset$). Then the element 
\[ \zeta :=\beta_{\fa_1 \cdot \fp_1} \hotimes \dots \hotimes \beta_{\fa_l \cdot \fp _l }=(\beta_{\fp_1} + \beta_{\fa_1}) \hotimes \dots \hotimes (\beta_{\fp_l} + \beta_{\fa_l}) \in \K_*(C^*_rP_K^1) \]
satisfies $\iota_*(\zeta)=h_K^1\cdot j_0(\zeta) \in \K_*(\sB_K^\emptyset )$ by Lemma \ref{lem:finite}. Since $\zeta$ is written as $\beta_F^\emptyset  + \sum_{i=1}^N c_i\beta_{G_i}^\emptyset $ by $c_i \in \bZ$ and finite subsets $G_i$ of $\cP_K$ satisfying $F \not \subset G_i$ (for $i=1,\dots , N$), Lemma \ref{lem:bdry} implies that 
\begin{align*}
D_K^F(\iota_*(\zeta ))=&D_K^F \big(\Val_{K*}^\emptyset \big(\beta_F^\emptyset + \sum_{i=1}^N c_i\beta_{G_i}^\emptyset \big) \big) \\
=&D_K^F(\Val_{K*}^\emptyset (\beta_F^\emptyset)) = (-1)^l \Val_{K*}^F (\beta_\emptyset^F) =  \iota_* [1_{C^*_rJ_K^F}].
\end{align*}
Consequently we get $(\tau _K^F \circ D_K^F)((-1)^{l} j_0(\zeta))=1$.
\end{proof}

Now we assume that there is a family of isomorphisms $\varphi^F \colon \K_*(B_K^F) \to \K_*(B_L^{\chi(F)})$ as in the condition (6) of Theorem \ref{thm:main}. Let $\hat{\varphi}^F:=\xi_L^{\chi(F)} \circ \varphi^F \circ (\xi_K^F)^{-1} $. By the commutativity of the diagrams (\ref{form:diagval}) and (\ref{form:diagphi}) and the definition of $D_K^F$, the diagrams 
\begin{align}
\begin{gathered}
\xymatrix{\K_*(C^*_rJ_K) \ar[r]^{D_\val^F} \ar[d]^{\pi _K^{\emptyset *} \rtimes J_K} & \K_{*+l}(C^*_rJ_K^F) \ar[d]^{\pi^{F*}_K \rtimes J_K^F} \\
\K_*(\sB_K^{\emptyset }) \ar[r]^{D_K^F} & \K_{*+l}(\sB_K^{F}) ,} \ \ 
\xymatrix{\K_*(\sB_K^{\emptyset }) \ar[r]^{D_K^F} \ar[d]^{\hat{\varphi}^\emptyset } & \K_{*+l}(\sB_K^{F}) \ar[d]^{\hat{\varphi}^F} \\
\K_*(\sB_L^{\emptyset }) \ar[r]^{D_L^{\chi(F)}} & \K_{*+l}(\sB_L^{\chi(F)}) }
\label{diag:D}
\end{gathered}
\end{align}
also commute.
\begin{prop}\label{prp:narrow}
Assume the condition (6) of Theorem \ref{thm:main}. 
\begin{enumerate}
\item The isomorphism $j_\chi \colon J_K \to J_L$ induced from $\chi$ restricts to an isomorphism $P_K^1 \to P_L^1$.
\item The diagram
\[ \xymatrix@C=4em{
\K_*(C^*_rJ_K) \ar[r]^{(\pi _K^{\emptyset *} \rtimes J_K)_*} \ar[d] ^{(j_\chi) _*} & \K_*(\sB_K^\emptyset ) \ar[d]^{\hat{\varphi}^\emptyset } \\
\K_*(C^*_rJ_L) \ar[r]^{(\pi _L^{\emptyset *} \rtimes J_L)_*} & \K_* (\sB_L^\emptyset ) 
}\]
commutes.
\end{enumerate}
\end{prop}
\begin{proof}
By Lemma \ref{lem:bdry} and Lemma \ref{lem:tau}, we have
\[ (\tau_K^F \circ D_K^F)(\beta_{F'}^\emptyset) = \left\{ \begin{array}{ll} h_K^1 & \text{ if $F=F'$,} \\ 0 & \text{ otherwise.} \end{array}\right. \]
Therefore, the homomorphism $\tilde{\Psi}_K$ from $\K_*(\sB_K^\emptyset )$ to $\lwedge ^* J_K$ determined by
\[ \tilde{\Psi} _K(x):= \sum _{F} (-1)^{|F|} (\tau_K^F \circ D_K^F) (x) \cdot \beta_F^\emptyset , \]
where $F$ runs over all finite subsets of $\cP_K$, satisfies that the composition 
\[ \lwedge ^* J_K \cong \K_*(C^*_rJ_K) \xrightarrow{(\pi_K^{\emptyset *} \rtimes J_K)_*} \K_*(\sB_K^\emptyset ) \xrightarrow{\tilde{\Psi}_K} \lwedge ^* J_K  \]
is the multiplication by $h_K^1$.  
Comparing it with Lemma \ref{lem:finite} we get $\tilde{\Psi}_K = \iota _* \circ( j_0)^{-1}$, where $\iota_*$ and $j_0$ are as in Lemma \ref{lem:finite} for the inclusion $P_K^1 \subset J_K$. Therefore we get $\Img \tilde{\Psi}_K = \lwedge^* P_K^1$ and in particular 
\begin{align}
\Img \tilde{\Psi}_K \cap \lwedge ^1 J_K=  P_K^1.\label{form:Psi}
\end{align}

Consider the diagram
\begin{align}
\begin{gathered}
\xymatrix@C=4em{
\K_*(C^*_rJ_K) \ar[r]^{(\pi _K^{\emptyset *} \rtimes J_K)_*} \ar[d] ^{(j_\chi) _*} & \K_*(\sB_K^\emptyset ) \ar[r]^{\tilde{\Psi}_K} \ar[d]^{\hat{\varphi}^\emptyset } & \lwedge ^* J_K \ar[d]^{\wedge ^* j_\chi} \\ 
\K_*(C^*_rJ_L) \ar[r]^{(\pi _L^{\emptyset *} \rtimes J_L)_*} &\K_*(\sB_L^\emptyset ) \ar[r]^{\tilde{\Psi}_L} & \lwedge ^* J_L 
}\label{form:diagPsi} 
\end{gathered}
\end{align}
The right square commutes by the commutativity of the right diagram in (\ref{diag:D}) and the uniqueness of $\tau_K^F$ in Lemma \ref{lem:tau}. Hence we get (1) by (\ref{form:Psi}). In particular we get $h_K^1=h_L^1$, which shows that the outer square of (\ref{form:diagPsi}) commutes. Since $\tilde{\Psi}_K$ and $\tilde{\Psi}_L$ are injective, a diagram chasing shows that the left square also commutes.
\end{proof}

\begin{proof}[Proof of Theorem \ref{thm:main}]
The steps (1)$\Rightarrow$(2)$\Rightarrow$(4), (1)$\Rightarrow$(3)$\Rightarrow$(4) are obvious. (4) $\Rightarrow$ (5) $\Rightarrow$ (6) is explained at the last paragraph of Subsection \ref{section:2.3}.

We show (6)$\Rightarrow$(1). In the proof, for a finite subset $F$, we omit $\chi$ and use the same symbol $F$ for its image in $\cP_L$ for simplicity of notation. Consider the diagram

\[ 
\xymatrix@C=3em{
\K_*(C^*_rJ_K) \ar[rrr]^{\pi _K^{\emptyset *} \rtimes J_K} \ar[ddd]^{(j_\chi) _*} \ar[rd]^{D_\val^F} \ar@{}[rrrd]|{(1)} \ar@{}[rddd]|{(2)} &&& \K_*(\sB_K^\emptyset ) \ar[ddd]^{\hat{\varphi}^\emptyset } \ar[ld]^{D_K^{F}} \\
&\K_{*+l}(C^*_rJ_K^F) \ar[r]^{\pi _K^{F *} \rtimes J_K^F} \ar[d] ^{(j_\chi) _*} \ar@{}[rd]|{(5)} & \K_{*+l}(\sB_K^F ) \ar[d]^{\hat{\varphi}^F}& \\
&\K_{*+l}(C^*_rJ_L^F) \ar[r]_{\pi _L^{F*} \rtimes J_L^{F}} & \K_{*+l} (\sB_L^{F} ) & \\
\K_*(C^*_rJ_L ) \ar[rrr]^{\pi _L^{\emptyset *} \rtimes J_L} \ar[ru]^{D_\val^{F}} &&& \K_* (\sB_L^{\emptyset } ). \ar[lu]^{D_L^{F}} \ar@{}[lllu]|{(3)} \ar@{}[luuu]|{(4)}
}
\]
We have already proved the commutativity of the large outer square and diagrams (1), (3), (4) at Proposition \ref{prp:narrow} (2) and (\ref{diag:D}) respectively.  The diagram (2) also commutes by definition. Since $D_{\val}^F$ is surjective, a diagram chasing shows that the diagram (5) also commutes.

Finally, with the help of Lemma \ref{lem:Gpro} we can apply Corollary \ref{cor:isom} to get isomorphisms $G_K^F \cong G_L^F$ such that the diagram (\ref{diag:JG}) commutes, which concludes the theorem by Proposition \ref{prp:global}. 
\end{proof}

Lastly, we give two remarks. 
First, the proof of Theorem \ref{thm:main} actually gives a procedure for reconstruction of the semigroup action from K-theoretic data. This is a stronger result than a mere classification. 
Second, Theorem \ref{thm:main} does not mean that the isomorphism is reconstructed from $\K$-theoretic data. Indeed, if we have an automorphism $\varphi$ on $A_K$ as a C*-algebra over $2^{\cP_K}$ (e.g.,\ the action of $G_{K}^{\mathrm{ab}}$ induced from its action by multiplication on the second factor of $Y_K$), then we can apply the commutativity of the diagram (5) in the proof of Theorem \ref{thm:main} above for the family $\{ \varphi^F_* \colon \K_*(B_K^F) \to \K_*(B_K^F) \}$ to see that $\varphi^F_*$ are identity maps.

\appendix
\section{Another direct reconstruction of profinite completions from $\K$-theory\\ by Xin Li}\label{section:appendix}

Let us present an alternative approach to Corollary \ref{cor:isom}. Let $\Gamma \to G$ be a profinite completion, where $\Gamma$ is a countable free abelian group and $G$ is a profinite completion of $\Gamma$ which is topologically finitely generated, as in Section \ref{section:3}. Let $G = \varprojlim_n \Gamma / \Gamma_n$ with finite index subgroups $\Gamma_1 \supseteq \Gamma_2 \supseteq \Gamma_3 \supseteq \dotso$ of $\Gamma$. Let $\iota$ be the map $K_*(C^*\Gamma ) \to K_*(C(G) \rtimes \Gamma)$ induced by the canonical homomorphism $C^*\Gamma \to C(G) \rtimes \Gamma$. In the following, we give a concrete recipe allowing us to reconstruct $\Gamma \to G$ from $\iota$.

Write $\Gamma = \bigoplus_{i=1}^{\infty} \bZ e_i$, $\Sigma_s = \bigoplus_{i=1}^s \bZ e_i$. The basis $\left\{ e_i \right\}$ gives rise to a canonical isomorphism $\bZ^s \cong \Sigma_s$, which is the only one we use to identify $\Sigma_s$ with $\bZ^s$. Let $G_s = \varprojlim_n \Sigma_s / (\Sigma_s \cap \Gamma_n)$, and $p_s: \: \Sigma_s \to G_s$ is the canonical map. We have a canonical isomorphism $\bigoplus_{r=0}^s \bigwedge^r \Sigma_s \cong K_*(C^*\Sigma_s )$ as in Section \ref{section:3}, so that we can consider the composition $\iota_s: \: \bigwedge^{s-1} \Sigma_s \to K_*(C^*\Sigma_s) \to K_*(C^*\Gamma) \overset{\iota}{\longrightarrow} K_*(C(G) \rtimes \Gamma)$. Define
$$
  T_s := \left\{ x \in K_*(C(G) \rtimes \Gamma): \: \exists \, n \in \bZ, \, n > 0 \ {\rm with} \ n \cdot x \in \Img(\iota_s) \right\}.
$$
We have a canonical identification
\begin{equation}
\label{S=LS}
\Sigma_s \cong {\bigwedge}^{s-1} \Sigma_s,
\end{equation}
which is given as follows: The pairing $\Sigma_s \times \bigwedge^{s-1} \Sigma_s \to \bigwedge^s \Sigma_s \cong \bZ, \, (x,y) \mapsto x \wedge y$ gives an isomorphism $\bigwedge^{s-1} \Sigma_s \cong \Hom(\Sigma_s, \bZ)$, and we obtain a second isomorphism $\Hom(\Sigma_s, \bZ) \cong \Sigma_s$ via the pairing $\Sigma_s \times \Sigma_s \to \bZ, \, (x,y) \mapsto \left\langle x,y \right\rangle$, where $\left\langle \cdot,\cdot \right\rangle$ is the standard Euclidean inner product $\left\langle (x_i),(y_i) \right\rangle = \sum_i x_i \cdot y_i$. Let $i_s$ be the composite of $\iota_s$ with the isomorphism $\Sigma_s \cong \bigwedge^{s-1} \Sigma_s$ from \eqref{S=LS}, $i_s: \: \Sigma_s \cong \bigwedge^{s-1} \Sigma_s \hookrightarrow T_s$. 

\begin{thm}
There is an isomorphism $\varphi_s$ of $T_s$ with a subgroup $Q_s$ of $\bQ^s$ uniquely determined by requiring that $\varphi_s \circ i_s$ is the canonical isomorphism $\Sigma_s \cong \bZ^s$. Then $\varphi_s$ induces an isomorphism $T_s / i_s(\Sigma_s) \cong Q_s / \bZ^s$, again denoted by $\varphi_s$. Let $\psi_s$ be the composite $\psi_s: T_s / i_s(\Sigma_s) \overset{\varphi_s}{\longrightarrow} Q_s / \bZ^s \hookrightarrow \bR^s / \bZ^s \cong \widehat{\bZ^s} \cong \widehat{\Sigma_s}$. Here $\widehat{\Sigma_s}$ stands for Pontrjagin dual, and the isomorphism $\bR^s / \bZ^s \cong \widehat{\bZ^s}$ sends $\dot{x} \in \bR^s / \bZ^s$ to $\chi \in \widehat{\bZ^s}$, where $\chi(z) = e^{2 \pi i \left\langle x,z \right\rangle}$. Let $\widehat{\psi_s}: \: \Sigma_s \to \widehat{T_s / i(\Sigma_s)}$ be the dual map of $\psi_s$. Then there is a (unique) isomorphism $\omega_s: \: \widehat{T_s / i(\Sigma_s)} \cong G_s$ such that $\omega_s \circ \widehat{\psi_s} = p_s$. The canonical inclusions $\Sigma_s \hookrightarrow \Sigma$, $G_s \hookrightarrow G$ give rise to isomorphisms $\Gamma \cong \varinjlim_s \Sigma_s$ and $\varinjlim_s G_s \cong G$ such that $\Gamma \cong \varinjlim_s \Sigma_s \overset{\varinjlim_s p_s}{\longrightarrow} \varinjlim_s G_s \cong G$ is the original profinite completion $\Gamma \to G$, where the connecting maps $\Sigma_s \to \Sigma_{s+1}$ and $G_s \to G_{s+1}$ are the ones induced by the canonical inclusion $\Sigma_s \hookrightarrow \Sigma_{s+1}$.
\end{thm}
\begin{proof}
We first explain the last claim. By assumption, $G$ is topologically finitely generated. Hence for $s$ big enough, the canonical map $G_s \to G$ is an isomorphism, i.e., $\Sigma_s + \Gamma_n = \Gamma$ for all $n$, or equivalently, the canonical map $\Sigma_s / (\Sigma_s \cap \Gamma_n) \to \Gamma / \Gamma_n$ is surjective. So the sequence $G_s \to G_{s+1} \to \dotso$ becomes stationary. Thus, it suffices to show that for $s$ big enough (i.e., for $s$ such that $\Sigma_s + \Gamma_n = \Gamma$ for all $n$), there is $\omega_s$ such that $\omega_s \circ \widehat{\psi_s} = p_s$.

Let us fix $s$ with $\Sigma_s + \Gamma_n = \Gamma$ for all $n$. To simplify notation, we set $\Sigma := \Sigma_s$, $\iota := \iota_s$, $T := T_s$ and so on, i.e., we drop the index $s$. We have an isomorphism
\begin{eqnarray}
\label{K=lim}
  &&K_*(C(G) \rtimes \Gamma)  \\
  &\cong&  \varinjlim \left\{ K_*(C(\Gamma/\Gamma) \rtimes \Gamma) \to K_*(C(\Gamma/\Gamma_1) \rtimes \Gamma) \to \dotso \right\} \nonumber \\ 
  &\cong& \varinjlim \left\{ K_*(C^*\Gamma) \to K_*(C^*\Gamma_1) \to \dotso \right\} \nonumber.
\end{eqnarray}
Here we identify $K_*(C(\Gamma/\Gamma_n) \rtimes \Gamma)$ with $K_*(C^*\Gamma_n)$ by embedding $C^*\Gamma_n$ as a full corner into $C(\Gamma/\Gamma_n) \rtimes \Gamma$ as in the paragraph before Lemma \ref{lem:finite}. For every $n$, we have by Lemma \ref{lem:finite} a commutative diagram
\begin{equation}
\label{D1}
  \begin{tikzcd}
   & K_*(C^*\Gamma)
  \\
  K_*(C^*\Gamma) \ar["d_n \cdot"]{ur} \ar{r} & K_*(C^*\Gamma_n) \ar{u}
  \end{tikzcd}
\end{equation}
Here the horizontal map is the composite of the first $n$ structure maps in the inductive limit \eqref{K=lim}, the vertical map is induced by the canonical inclusion $\Gamma_n \hookrightarrow \Gamma$, and the map $d_n \cdot: \: K_*(C^*\Gamma) \to K_*(C^*\Gamma)$ is multiplication with $d_n = [\Gamma:\Gamma_n]$. If we now choose an isomorphism $\mu_n: \: \Gamma \cong \Gamma_n$, then we can expand \eqref{D1} to
\begin{equation}
\label{D2}
  \begin{tikzcd}
   & K_*(C^*\Gamma)
  \\
  K_*(C^*\Gamma) \ar["d_n \cdot"]{ur} \ar{r} \ar[bend right]{rr} & K_*(C^*\Gamma_n) \ar{u} & \ar["(\mu_n)_*"']{l} K_*(C^*\Gamma)
  \end{tikzcd}
\end{equation}
This gives maps $K_*(C^*\Gamma) \to K_*(C^*\Gamma)$ (the arrow in \eqref{D2} from the lower left to the lower right copy of $K_*(C^*\Gamma)$) such that $K_*(C(G) \rtimes \Gamma)$ can be identified with the inductive limit of $K_*(C^*(\Gamma)) \to K_*(C^*\Gamma) \to \dotso$.

Now we choose $\mu_n: \: \Gamma \cong \Gamma_n$ such that $\mu_n(\Sigma) = \Sigma \cap \Gamma_n$. This is possible as $\Sigma \cap \Gamma_n$ is a direct summand of $\Gamma_n$. As $\Sigma \cong \bZ^s$, the composition $\Sigma \overset{\mu_n \vert_{\Sigma}}{\longrightarrow} \Sigma \cap \Gamma_n \hookrightarrow \Sigma$ is given by a matrix $M_n$ with integer entries. Restricting the left copy of $K_*(C^*\Gamma)$ to $\bigwedge^{s-1} \Sigma$ in \eqref{D2}, we obtain
\begin{equation}
\label{D3}
  \begin{tikzcd}
   & \bigwedge^{s-1} \Sigma
  \\
  \bigwedge^{s-1} \Sigma \ar["d_n \cdot"]{ur} \ar{r} \ar[bend right]{rrr} & \bigwedge^{s-1} (\Sigma \cap \Gamma_n) \ar{u} & & \ar["\bigwedge^{s-1} (\mu_n \vert_{\Sigma})"']{ll} \ar["\bigwedge^{s-1} M_n"']{llu} \bigwedge^{s-1} \Sigma
  \end{tikzcd}
\end{equation}
After deleting $\bigwedge^{s-1} (\Sigma \cap \Gamma_n)$ and identifying $\bigwedge^{s-1} \Sigma$ with $\Sigma$ as in \eqref{S=LS}, \eqref{D3} becomes
\begin{equation}
\label{D4}
  \begin{tikzcd}
   & \Sigma
  \\
  \Sigma \ar["d_n \cdot"]{ur} \ar{rr} & & \ar["(M_n^{\rm adg})^t"']{lu} \Sigma
  \end{tikzcd}
\end{equation}
Here $M_n^{\rm adg}$ is the adjugate matrix of $M_n$, uniquely determined by $M_n^{\rm adg} \cdot M_n = \det(M_n) \cdot I$ ($I$ being the identity matrix). $(M_n^{\rm adg})^t$ is the transpose of $M_n^{\rm adg}$. As $d_n = [\Gamma:\Gamma_n] = [\Sigma:\Sigma \cap \Gamma_n] = \det(M_n)$, the missing map $\Sigma \to \Sigma$ (from the lower left to the lower right copy of $\Sigma$ in \eqref{D4}) must be given by $M_n^t$. Hence we can complete \eqref{D4} to
\begin{equation}
\label{D5}
  \begin{tikzcd}
   & \Sigma
  \\
  \Sigma \ar["d_n \cdot"]{ur} \ar["M_n^t"]{rr} & & \ar["(M_n^{\rm adg})^t"']{lu} \Sigma
  \end{tikzcd}
\end{equation}
Thus $T$ is the inductive limit of the stationary inductive system $\Sigma \to \Sigma \to \dotso$ where the composition of the first $n$ structure maps is given by $M_n^t: \: \Sigma \to \Sigma$ (and this determines the inductive limit). It is easy to see that $T$ contains this inductive limit. Conversely, that $T$ is contained in this inductive limit follows from the fact that $\bigwedge^{s-1} \Sigma$ is a direct summand in $K_*(C^*\Gamma)$, so that if $n \cdot x$ lies in $\bigwedge^{s-1} \Sigma \subseteq K_*(C^*\Gamma)$ for some $n \in \bZ$, $n > 0$ and $x \in K_*(C^*\Gamma)$, then $x$ itself must lie in $\bigwedge^{s-1} \Sigma$. Now, there is only one way to complete the diagram
\begin{equation}
\label{phi0}
  \begin{tikzcd}
  \Sigma \ar{d} \ar{r} & \Sigma \ar{r} & \dotso
  \\
  \bQ^s & &
  \end{tikzcd}
\end{equation}
if we start with $\Sigma \cong \bZ^s \hookrightarrow \bQ^s$ as our first vertical map and want that the diagram commutes (the first row in \eqref{phi0} is the inductive system from above giving rise to $T$). The completed diagram is given by 
\begin{equation}
\label{phi}
  \begin{tikzcd}
  \Sigma \ar{d} \ar{r} & \Sigma \ar["M_1^{-t}"', bend left]{ld} \ar{r} & \dotso \ar{r} & \Sigma \ar["{M_n^{-t}}"', bend left]{llld} \ar{r} & \dotso
  \\
  \bQ^s & & & &
  \end{tikzcd}
\end{equation}
Here $M_n^{-t}$ is the inverse of $M_n^t$ (as a matrix over $\bQ$). Now the desired isomorphism $\varphi: \: T \cong Q = \bigcup_n M_n^{-t} \bZ^s \subseteq \bQ^s$ arises as the inductive limit in \eqref{phi}. $\psi$ becomes $T / i(\Sigma) \cong Q / \bZ^s = \bigcup_n \left( M_n^{-t} \bZ^s / \bZ^s \right) \hookrightarrow \bR^s / \bZ^s \cong \widehat{\bZ^s} \cong \widehat{\Sigma}$. For fixed $n$, the image of $M_n^{-t} \bZ^s / \bZ^s \hookrightarrow \bR^s / \bZ^s \cong \widehat{\bZ^s}$ is $\widehat{\bZ^s / M_n \bZ^s} \subseteq \widehat{\bZ^s}$ because
$$
  M_n^{-t} \bZ^s = \left\{ x \in \bR^s: \: \left\langle x,z \right\rangle \in \bZ \, \forall \, z \in M_n \bZ^s \right\}.
$$
Hence $\psi$ is given by
$$
  T / i(\Sigma) \cong \varinjlim_n \widehat{\bZ^s / M_n \bZ^s} = \varinjlim_n \widehat{\Sigma / (\Sigma \cap \Gamma_n)} \hookrightarrow \widehat{\Sigma}.
$$
Therefore, up to composing with an isomorphism called $\omega$, $\widehat{\psi}$ is given by $\Sigma \to \varprojlim_n \Sigma / (\Sigma \cap \Gamma_n)$, as claimed.
\end{proof}

\begin{rmk}
If $\Gamma$ itself is finitely generated, the proof becomes even easier, as we can take $\Sigma = \Gamma$.
\end{rmk}

\bibliographystyle{plain}
\bibliography{BostConnes}
\end{document}